\documentclass[10pt,a4paper]{amsart}

\usepackage{amssymb}

\usepackage{amscd}
\usepackage{hyperref}

\usepackage{euscript,color}

\newcommand{\nn}[1]{(\ref{#1})}

\newcommand{\Wp}[2]{W^{#1}{}_{#2}}
\newcommand{\Zp}[2]{Z_{#1}{}^{#2}} 

\newcommand{\bp}{\boldsymbol{p}}

\newcommand{\cB}{\mathcal{B}}

\newcommand{\lpl}{
  \mbox{$
  \begin{picture}(12.7,8)(-.5,-1)
  \put(2,0.2){$+$}
  \put(6.2,2.8){\oval(8,8)[l]}
  \end{picture}$}}

\def\frak{\mathfrak}

\let\phi\varphi

\newcommand{\si}{\sigma}

\def\Rho{\mbox{\textsf{P}}}

\newcommand{\End}{\operatorname{End}}

\newcommand{\D}{\mathbb{D}}

\newcommand{\nd}{\nabla}

\newcommand{\ce}{\mathcal{E}}
\newcommand{\cT}{\mathcal{T}}

\newcommand{\rpl}                         % +) or <+
{\mbox{$
\begin{picture}(12.7,8)(-.5,-1)
\put(0,0.2){$+$}
\put(4.4,3.1){\oval(8,8)[r]}
\end{picture}$}}

\newcounter{theorem}
\newtheorem{thm}[theorem]{Theorem}
\newtheorem*{thm*}{Theorem \thesubsection}
\newtheorem{lemma}[theorem]{Lemma}
\newtheorem{prop}[theorem]{Proposition}
\newtheorem{cor}[theorem]{Corollary}
\newtheorem*{lemma*}{Lemma \thesubsection}
\newtheorem*{prop*}{Proposition \thesubsection}
\newtheorem*{cor*}{Corollary \thesubsection}

\theoremstyle{definition}

\newtheorem*{definition*}{Definition \thesubsection}

\newtheorem*{example*}{Example \thesubsection}
\theoremstyle{remark}
\newtheorem{remark}[theorem]{Remark}
\newtheorem*{remark*}{Remark \thesubsection}

\def\sideremark#1{\ifvmode\leavevmode\fi\vadjust{\vbox to0pt{\vss% the remark
 \hbox to 0pt{\hskip\hsize\hskip1em%                          will appear only
 \vbox{\hsize3cm\tiny\raggedright\pretolerance10000%          on the side
  \noindent #1\hfill}\hss}\vbox to8pt{\vfil}\vss}}}%
\begin{document}

\title
[Invariant Prolongation of the Killing tensor equation]
{Invariant prolongation of the Killing tensor equation}
\author[A.~R.~Gover]{A.~Rod Gover}
\address{A.R.G.: Department of Mathematics\\
The University of Auckland\\
Private Bag 92019\\
Auckland 1142\\
New Zealand}
\email{r.gover@auckland.ac.nz}

\author[T.~Leistner]{Thomas Leistner}
\address{T.L.: School of Mathematical Sciences\\University of Adelaide\\SA 5005\\Australia
%\telephone: +61 (0)8 83136401,
%fax: +61 (0)8 83133696
}
\email{thomas.leistner@adelaide.edu.au}

%\email{thomas.leistner@adelaide.edu.au}
%
%\author[Daniel Schliebner]{Daniel Schliebner} \address[Schliebner]{Humboldt-Universit\"{a}t zu Berlin\\
%Institut f\"{u}r Mathematik\\Rudower Chaussee 25\\12489~Berlin\\ Germany}
%\email{schliebn@mathematik.hu-berlin.de}

%\thanks{
%{\bf \today\ at \xxivtime}}

\begin{abstract}
The Killing tensor equation is a first order differential equation on
symmetric covariant tensors that generalises to higher rank the usual
Killing vector equation on Riemannian manifolds. We view this more
generally as an equation on any manifold equipped with an affine
connection, and in this setting derive its prolongation to a linear
connection. This connection has the property that parallel sections
are in 1-1 correspondence with solutions of the Killing
equation. Moreover this connection is projectively invariant and  is
derived entirely using the projectively invariant tractor calculus which reveals also further invariant structures linked to the prolongation. 
  \end{abstract}

%% \author{ Maciej Dunajski and A.\ Rod Gover}

%% \address{M.D.: Faculty of Mathematics\\
%% \\
%% A.R.G.:Department of Mathematics\\
%%   The University of Auckland\\
%%   Private Bag 92019\\
%%   Auckland 1142\\
%%   New Zealand} 
%% \email{m.??}
%% \email{r.gover@auckland.ac.nz}

%% \begin{abstract}
%%   Consider a manifold 

%% \end{abstract}


\thanks{ARG gratefully acknowledges support from the Royal
  Society of New Zealand via Marsden Grant 16-UOA-051. TL was  partially supported by the grant 346300 for IMPAN from the Simons Foundation and the matching 2015-2019 Polish MNiSW fund}

\maketitle

\section{Introduction}

On a Riemannian manifold $(M,g)$ a tangent vector field $k\in
\frak{X}(M)$ is an infinitesimal automorphism (or symmetry) if the Lie derivative of the metric $g$ in direction of $k$ vanishes.
In terms of the Levi-Civita connection $\nabla=\nabla^g$, this
may be written as
\begin{equation}\label{one}
\nabla_{(a}k_{b)}=0
\end{equation}
where we use an obvious abstract index notation, $k_a=g_{ab}k^b$, and the
$(ab)$ indicates symmetrisation over the enclosed indices. This {\em
  Killing equation} is generalised to higher rank $r\geq 1$ by the
{\em Killing tensor equation} equation
\begin{equation}\label{hKill}
\nabla_{(a}k_{b \cdots c)}=0 
  \end{equation}
 where $k_{b\cdots c}$ is a symmetric tensor, that is $k \in
 \Gamma(S^r T^*M)$ and again $(ab \cdots c)$ indicates symmetrisation
 over the enclosed indices. Solutions of this, so-called { {\em
     Killing tensors}, are important for treatment of separation of
   variables \cite{AnderssonBlue15,Kalnins86,Miller77book,SchobelVeselov15}, higher
   symmetries of the Laplacian and similar operators
   \cite{AnderssonBackdahlBlue14,DurandLinaVinet88,eastwood05,GoverSilhan12,LevasseurStafford17,MichelSombergv-Silhan17}, and for the theory of
   integrable systems, and superintegrability \cite{CarigliaGibbonsHoltenHorvathyKosinski14,DuvalValent05,De-BieGenestLemayVinet17,KalninsKressWinternitz02,KalninsKressMiller05}. Partly these applications arise because a solution of
   \nn{hKill} (for any $r$) provides a first integral along geodesics:
   if $\gamma:I\to M$ is a geodesic (where $I\subset \mathbb{R}$ is an
   interval) and $u:=\dot\gamma$ is the velocity of this then
   $\nabla_uu=0$ and therefore by dint of \nn{hKill} the function
   $k_{b\cdots c}u^b\cdots u^c$ is constant along $\gamma$.

   In dimensions $n\geq 2$ (which we assume throughout) the equation
   \nn{hKill} is an overdetermined finite type linear partial
   differential equation. This means, in particular, that it is
   equivalent to a linear connection on a system that involves the
   Killing tensor $k$ but also additional variables, the {\em
     prolonged system} \cite{BransonCapEastwoodGover06,Spencer69}.  For example for
   equation (\ref{one}) above this prolonged system is very easily found to be
   \begin{equation}\label{eg}
\overline{\nabla}_a\left(\begin{array}{c}k_c\\ \mu_{bc}\end{array}\right) = \left(\begin{array}{c}\nabla_a k_b- \mu_{ab}\\ \nabla_a\mu_{bc} - R^{\phantom{b}}_{bc}{}^d{}_{a}k_d ,\end{array}\right)
   \end{equation}
   where $ R^{\phantom{b}}_{bc}{}^d{}_{a}$ is the curvature of $\nabla$
   (see Section \ref{curvedrk1sec} below).
   In general such
   prolonged systems are not unique, but for any such connection its
   parallel sections correspond 1-1 with solutions of the original
   equation (\nn{hKill} in this case). Thus, on connected manifolds,
   the rank of the prolonged systems gives an upper bound on the
   dimension of the space of solutions and curvature of the given
   connection can lead to obstructions to solving the equation, see
   e.g. \cite{BryantDunajskiEastwood09,GoverMacbeth14,gover-nurowski04}.

   Two affine connections $\nabla$ and $\nabla'$ are said to be {\em
     projectively equivalent} if they share the same unparametrised
   geodesics. Connections differing only by torsion are projectively
   related, and we will lose no generality in our work here if we
   restrict to torsion free connections, which we do henceforth.  An
   equivalence class of $\bp=[\nabla]$ of such projectively related
   torsion-free connections is called a {\em projective structure} and
   a manifold $M^{n \geq 2}$ equipped with such a structure is called
   a {\em projective manifold}.  An important but not fully exploited
   feature of the equation \nn{hKill} is that it is {\em projectively
     invariant}. This will be explained fully in Section \ref{ptrac},
    but at this stage it will
   suffice to say the following. First when we introduced \nn{hKill}
   above, $\nabla $ denoted the Levi-Civita connection of a metric, but the equation
   makes sense and is important for any affine connection $\nabla$,
   and it is in this setting that we now study it.  Next the
   projective invariance means that the equation \nn{hKill} has a
   certain insensitivity and, in particular, descends to a well
   defined equation on a projective manifold $(M,\bp)$.

On a general projective manifold $(M,\bp)$ there is no distinguished
affine connection on $TM$. However there is a distinguished
projectively invariant connection $\nabla^{\cT}$ on a vector bundle
$\cT$ that extends (a density twisting of) the tangent bundle $TM$:
\begin{equation}\label{ceuler}
0\to \ce(-1) \stackrel{X}{\to}\cT\to TM\otimes \ce(-1) \to 0
\end{equation}
where $\ce(-1)$ is a natural real oriented line bundle defined in
Section \ref{back} below. This is the normal projective tractor
connection and it (or the equivalent Cartan connection) provides the
basic tool for invariant calculus on projective manifolds. An
important feature of this connection is that it is on a low rank
bundle (i.e. $\operatorname{dim}(TM)+1$) that is simply related to the
tangent bundle. The tractor calculus is recalled in Section \ref{ptrac}.

   For most applications that one can imagine it makes sense then to
   seek a prolongation of \nn{hKill} that is itself  a projectively
   invariant connection. For example, if this can be found, then its
   curvature simultaneously constrains solutions for the entire class
   of projectively related connections. In fact such a connection
   exists. The equations \nn{hKill} is an example of a first BGG
   equation and arises as a special case of the very general theory of
   Hammerl et al.\ in \cite{HammerlSombergSoucekSilhan12} (see also
   \cite{HammerlSombergSoucekSilhan12inv}). That theory describes an algorithm for
   producing an invariant connection giving the prolonged system for
   any of the large class of BGG equations (and we refer the reader to
   that source for the meaning of these terms) and in this sense is
   very powerful. Although the algorithm of \cite{HammerlSombergSoucekSilhan12} produces
   in the end an invariant connection it proceeds through stages that
   break the invariance of the given equation. For example in treating
   \nn{hKill} the steps of the algorithm are not projectively
   invariant. Moreover beyond the case of rank 1 the explicit
   treatment of \nn{hKill} using this algorithm seems practically
   intractible due to the number of steps involved. Finally although
   the construction of \cite{HammerlSombergSoucekSilhan12} is strongly linked to the
   calculus of the normal tractor connection (of
   \cite{bailey-eastwood-gover94,cap/gover02,cap-slovak-book09}) the connection finally obtained
   is not easily linked to the normal tractor connection.
   %%, and thus is
   %% a connection on a very high rank bundle without an easy link to invariant lower r

The aim of this article is to produce an alternative invariant prolongation
procedure that is simple, conceptual, explicit, and that reflects the
invariance properties of the original equations. It is well known that
for the projective BGG equations the normal tractor connection easily
recovers the required prolongation in the case that the structure is
projectively flat (i.e., the projective tractor/Cartan connection is
flat). A motivation is to be able to produce the explicit curvature
correction terms that modify the normal tractor connection to deal
with general solutions on a projectively curved manifold. An explicit knowledge of these terms will enable us to deduce properties of the prolongation and so properties of solutions in general. 
 We develop here a projectively invariant prolongation of the equation
   \nn{hKill} for each $r\geq 1$.
   This uses at all stages the calculus
   of the normal projective tractor conection $\nabla^{\cT}$ (as in
   \cite{bailey-eastwood-gover94}). The result is a connection on a certain projective
   tractor bundle (a tensor part of a power of the dual $\cT^*$ to
   $\cT$) that differs from the normal tractor connection by the
   algebraic action of a tractor field that is projectively invariant
   and produced in a simple way from the curvature of the normal
   tractor connection and iterations of a projectively invariant
   operator on this. An advantage is that the construction and
   calculation uses projectively invariant tools, and at all stages
   the link to the very simple normal tractor connection is manifest.
As an immediate application this approach typically simplifies the computation of integrability conditions, see Remark \ref{int-app} and in particular equation~\eqref{int-cond}.

   A tensorial approach to prolonging the Killing equation has been
   developed for arbitrary rank in \cite{Wolf98} (see also
   \cite{Collinson71}).  Concerning our results for the projectively
   flat case in Section \ref{pflat} there are necessarily some strong
   links to the prolongation approach of \cite{MichelSombergv-Silhan17}. However our route
   to the prolongation is very different and it is this that is
   important for the development of the curved theory.

   In fact there is considerable information in some of the
   preliminary results along the way in our treatment. For example
   each Killing equation is captured in the very simple tractor
   equation of Proposition \ref{pp}. This is part of a rather general
   picture which suggests that the theory here should generalise
   considerably. (In fact aspects of our treatment here were inspired
   by the conformally invariant prolongation of the conformal Killing
   equation via tractors in \cite[Proposition 2.2]{Gover}.)  This will
   be taken up in subsequent works. The Proposition \ref{pp} also may
   interpreted as showing that solutions of the Killing tensor equation
   on $(M,\bp)$ correspond in a simple way to Killing tensors for the
   canonical affine connection on the Thomas cone over $(M,\bp)$; the
   Thomas cone is discussed in e.g. \cite{CGH,cap-slovak-book09}.

Throughout we use Penrose's abstract index notation. As mentioned
above $(ab\cdots c)$ indicates symmetrisation over the enclosed
indices, while $[ab\cdots c]$ indicates skewing over the enclosed
indices. Then  $\ce$
   is used to denote the trivial bundle, and for example $\ce_{(abc)}$ is the bundle of covariant symmetric 3-tensors $S^3T^*M$.

\section{Background}\label{back}

\subsection{Conventions for affine geometry}\label{notat}

Let $(M,\nabla)$ be an  affine manifold (of dimension $n\geq
2$), meaning that $\nabla$ is a torsion-free affine connection.
The curvature
 $$R_{ab}{}^c{}_d\in\Gamma(\Lambda^2 T^*M \otimes TM\otimes T^*M )$$
of the
connection $\nabla$ is given by
$$
[\nabla_a,\nabla_b]v^c=R_{ab}{}^c{}_d  v^d , \qquad v\in \Gamma(TM).
$$
The Ricci curvature is defined by $R_{bd} =R_{cb}{}^c{}_d $. 

On an affine manifold the
trace-free part $W_{ab}{}^c{}_d$ of the curvature $R_{ab}{}^c{}_d$ is
called the {\em projective Weyl curvature} and we have
\begin{equation}\label{decp}
R_{ab}{}^c{}_d= W_{ab}{}^c{}_d +2 \delta^c_{[a}\Rho_{b]d}+\beta_{ab}\delta^c_d,
\end{equation}
where $\beta_{ab}$ is skew and $\Rho_{ab}$ is called the {\em projective Schouten tensor}. That $W_{ab}{}^c{}_d$ is trace-free means exactly that $W_{ab}{}^a{}_d=0$ and $W_{ab}{}^d{}_d=0$.
Since $\nabla$ is torsion-free the Bianchi symmetry
$R_{[ab}{}^c{}_{d]}=0$ holds, whence
$$
\beta_{ab}=-2\Rho_{[ab]} \qquad \mbox{and} \qquad (n-1)\Rho_{ab} = R_{ab}+\beta_{ab}.
$$

As we shall see below the curvature decomposition \nn{decp} is useful
in projective differential geometry.

First some further notation. On a smooth $n$-manifold $M$ the bundle
$\mathcal{K}:=(\Lambda^{n} TM)^2$ is an oriented line bundle and thus
we can take correspondingly oriented roots of this. For projective
geometry a convenient notation for these is as follows:
given $w\in
\mathbb{R}$ we write
\begin{equation} \label{pdensities}
\ce(w):=\mathcal{K}^{\frac{w}{2n+2}} . 
\end{equation}
\newcommand{\cK}{\mathcal{K}} Of course the affine connection $\nabla$
acts on $\Lambda^{n} TM$ and hence on the {\em projective density
  bundles} $\ce(w).$ As a point of notation, given a vector bundle
$\cB$ we often write $ \cB(w) $ as a shorthand for $\cB\otimes \ce(w)$.

\subsection{Projective geometry and tractor calculus}\label{ptrac}

Two affine torsion-free connections $\nabla' $ and $\nabla$ are
projectively equivalent, that is they share the same unparametrised
geodesics, if and only if there some $\Upsilon \in \Gamma(T^*M)$ s.t.
\begin{equation}\label{trans}
\nabla'_a v^b=\nabla_a v^b +\Upsilon_a v^b+ \Upsilon_c v^c\delta_a^b
\end{equation}
for all $v\in \Gamma(T^*M)$. This implies that on sections of $\ce(w)$ we have
$$
\nabla'_a \tau =\nabla_a\tau +w\Upsilon_a \tau,
$$
while on sections of $T*M$,
$$
\nabla'_a u_b=\nabla_a u_b -\Upsilon_a u_b-  \Upsilon_b u_a 
$$
It follows at once that on $k_{a_1\cdots a_k}\in S^k T^*M(2r)$ we have 
$$
\nabla'_{(a_0} k_{a_1\cdots a_k)}=\nabla_{(a_0} k_{a_1\cdots a_k)}.
$$
Thus for $k\in S^k T^*M(2r)$ the Killing equation \eqref{hKill} is {\em
  projectively invariant} and descends to a well defined equation on
$(M,\bp)$, where $\bp=[\nabla]=[\nabla']$, the projective equivalence
class of $\nabla$.

On a general projective $n$-manifold $(M,\bp)$ there is no distinguished
connection on $TM$. However there is a projectively invariant
connection on a related rank $(n+1)$ bundle $\cT$. This is the
projective tractor connection that we now describe.

Consider the first jet prolongation
$J^1\ce(1)\to M$ of the density bundle $\ce(1)$. (See for example
\cite{Palais65} for a general development of jet bundles.)
There is a canonical bundle map called the {\em jet projection map}
$J^1\ce(1)\to\ce(1)$, which at each point is determined by the map
from 1-jets of densities to simply their evaluation at that point, and
this map has kernel $T^*M (1)$.  We write $\cT^*$, or an in an
abstract index notation $\ce_A$, for $J^1\ce(1)$ and $\cT$ or $\ce^A$
for the dual vector bundle. Then we can view the jet projection as a
canonical section $X^A$ of the bundle $\ce^A(1)$. Likewise,
the inclusion of the kernel of this projection can be viewed as a
canonical bundle map $\ce_a(1)\to\ce_A$, which we denote by
$Z_A{}^a$. Thus the jet exact sequence (at 1-jets) is written in this
notation as
\begin{equation}\label{euler}
0\longrightarrow \ce_a(1)\stackrel{Z_A{}^a}{\longrightarrow} \ce_A \stackrel{X^A}{\longrightarrow}\ce(1)\longrightarrow 0.
\end{equation}
We write $\ce_A=\ce(1)\lpl \ce_a(1)$ to summarise the composition
structure in \nn{euler} and $X^A\in \Gamma(\ce^{A}(1))$, as defined in
\nn{euler}, is called the {\em canonical tractor} or {\em position tractor}. Note the sequence \nn{ceuler} is simply the dual to \nn{euler}.

As mentioned above, 
 any connection $\nabla \in \bp$
determines a connection on $\ce(1)$. On the other hand,  by definition,
a connection on $\ce(1)$  is precisely a splitting
of the 1-jet sequence \nn{euler}. 
Thus given such a choice we have the direct sum
decomposition $\ce_A \stackrel{\nabla}{=} \ce(1)\oplus \ce_a(1) $
and we write
 \begin{equation}\label{split}
 Y_A:\ce(1) \to \ce_A \qquad \mbox{and} \qquad W^A{}_a: \ce_A\to \ce_a(1), 
 \end{equation}
for the bundle maps giving this splitting of \nn{euler}; so
$$
 X^A Y_A=1, \qquad  Z_A{}^b W^A{}_a=\delta^b_a, \qquad \mbox{and} \qquad Y_A W^A{}_a=0. 
 $$
 By definition $X$ and $Z$ are projectively invariant. The formulae
 for how $Y_A$ and $W^A_a$ transform when $\nabla$ is replaced by
 $\nabla'$, is in \eqref{trans}, is easily deduced and can be found in
 \cite{bailey-eastwood-gover94}.

\newcommand{\bX}{\mathbb{X}}

With
respect to a splitting \nn{split}
 we define a connection on $\cT^*$ by
\begin{equation}\label{pconn}
\nabla^{\mathcal{T}^*}_a \binom{\si}{\mu_b}
:= \binom{ \nabla_a \si -\mu_a}{\nabla_a \mu_b + \Rho_{ab} \si}.
\end{equation}
Here $\Rho_{ab}$ is the projective Schouten tensor of $\nabla\in \bp$, as introduced earlier.
It turns out that \nn{pconn} is
independent of the choice $\nabla \in \bp$, and so
$\nabla^{\mathcal{T}^*}$ is determined canonically by the projective
structure $\bp$. We have followed the construction of \cite{bailey-eastwood-gover94,CapGoverMacbeth14}, but as mentioned in those sources
this {\em cotractor connection} is due to  T.Y. Thomas. Thus we shall
also term $\cT^*=\ce_A$ the {\em cotractor bundle}, and we note the dual
{\em tractor bundle} $\cT=\ce^A$ 
has canonically the dual {\em tractor connection}: in terms of a
splitting dual to that above this is given by
\begin{equation}\label{tconn}
\nabla^\cT_a \left( \begin{array}{c} \nu^b\\
\rho
\end{array}\right) =
\left( \begin{array}{c} \nabla_a\nu^b + \rho \delta^b_a\\
\nabla_a \rho - \Rho_{ab}\nu^b
\end{array}\right).
\end{equation}
Note that given a choice of $\nabla\in \bp$, by coupling with the
tractor connection we can differentiate tensors taking values in
tractor bundles and also weighted tractors. In particular
we have 
\begin{equation}\label{trids}
\nabla_aX^B=W^B{}_a,  \quad \nabla_a W^B{}_b=-\Rho_{ab} X^A  , \quad  \nabla_a Y_B =\Rho_{ab}Z_B{}^b,  \quad\mbox{and}\quad  \nabla_a Z_B{}^b = -\delta^b_a Y_B.
  \end{equation}

The curvature of the tractor connection is given by  
\begin{equation}\label{tractor_curvature}
\kappa_{ab}{}{}^C{}_{D}=W_{ab}{}^c{}_d W^C{}_c Z_{D}{}^d-C_{abd}Z_{D}{}^d X^C,
%\begin{pmmatrix}
%		 		0   & -C_{abc} \\
%		 		0 & W_{ab}{}^c{}_d
%		 	\end{pmatrix} .
\end{equation}
where $W_{ab}{}^c{}_d$ is the projective Weyl curvature, as above, and 
\begin{equation}\label{Cotton}
C_{abc}:= \nabla_a \Rho_{bc}-\nabla_b \Rho_{ac}
\end{equation} 
is called the {\em
  projective Cotton tensor}.

\newcommand{\cV}{\mathcal{V}}
\newcommand{\cU}{\mathcal{U}}
The projective Thomas-D operator is a first order projectively
invariant differential operator, or more  accurately family of such operators.
Given any tractor bundle $\cV$ (including the trivial bundle $\ce$) and any $w\in \mathbb{R}$ it provides an operator on the weighted tractor bundle $\cV(w)$
$$
\D: \cV(w) \to \cT^*\otimes \cV(w-1)
$$
given by
\begin{equation}\label{TD}
\D_A V= w Y_A V + Z_A{}^a\nabla_a V ,
\end{equation}
where $\nabla_a$ is the connection induced 
%\tedz{expalantion added, please check!} 
on the weighted bundle $\mathcal V$ from the tractor connection $\nabla_a^{\cT^*}$ and the connection on $\mathcal E(1)$ coming from a representative in $\bp$. 
%\tedz{ added smtg about twist}
Note that from this definition and \nn{trids}  follows
\begin{equation}\label{DX}
\D_A X^B= \delta^B_A, \qquad \mbox{and} \qquad X^A \D_A V =w V,  
\end{equation}
for $V\in \Gamma(\cV(w) )$.
Also from the definition it follows that $\D$ satisfies a Leibniz
rule, in that if $\cU(w)$ and $\cV(w')$ are tractor (or density)
bundles of weights $w$ and $w'$, respectively then for sections $U\in
\Gamma(\cU(w))$ and $V\in \cV(w')$ we have
$$
\D(U\otimes V)= (\D U) \otimes V+ U\otimes \D V.
$$
Thus from \nn{DX}, when commuting  $\D_A$ with the tensor product with $X^B$, we get the commutator identity
\begin{equation}\label{[DX]}
  [\D_A, X^B]=\delta_A{}^B .
  \end{equation}

In view of the last property, as an operator on weighted tractor fields,
the commutator $[\D_A,\D_B]$ is a ``curvature'' in that it acts algebraically. We will treat it this way by writing,
\begin{equation}\label{Wdef}
[\D_A,\D_B] V^C=W_{AB}{}^C{}_D V^D
\end{equation}
for its action on $V\in \gamma(\cT(w))$. For this reason and for convenience we will refer to $W_{AB}{}^C{}_D$ as the $W$-curvature.
Investigating this,   consider $\D$ on projective densities $\tau\in \Gamma(\ce(w))$
to form $\D_B\tau$. Using \nn{trids} we have
\begin{align*}
  \D_A\D_B\tau & =  (w-1)Y_A \D_B \tau + Z_A{}^a\nd_a \D_B\tau \\
  &= w(w-1)Y_AY_B\tau + 2(w-1)Y_{(A} Z^b_{B)}\nd_b\tau + Z^a_AZ^b_B \nd_a\nd_b\tau ,
\end{align*}
which we note is symmetric. Phrased alternatively, we have on sections
of density bundles
\begin{equation}\label{Dcomm}
[\D_A,\D_B]\tau=0 . 
\end{equation}
So $\D$ is ``torsion free'' in this sense, and from the Jacobi identity
we have at once the Bianchi identities
\begin{equation}\label{B12}
W_{[AB}{}^C{}_{D]}=0
\qquad \mbox{and}\qquad
\D_{[A}W_{BC]}{}^E{}_F=0 .
\end{equation}
To compute $W_{AB}{}^C{}_D$ it suffices to act on a section  $V\in \Gamma(\cT)$.
Note from \nn{trids}
$$
\D_A\D_B V^C = -Y_{A}\D_B V^C -Y_{B}\D_A V^C +  \Zp{A}{a}\Zp{B}{b} \nd_a\nd_b V^C .
$$
Thus
\begin{equation}\label{Wexp}
W_{AB}{}^C{}_D= \Zp{A}{a}\Zp{B}{b}\kappa_{ab}{}^C{}_D ,
\end{equation}
where $\kappa$ is the tractor curvature given above, and in particular
\begin{equation}\label{XW0}
X^AW_{AB}{}^C{}_D=X^BW_{AB}{}^C{}_D=X^DW_{AB}{}^C{}_D=0,
\end{equation}
as well as
\begin{equation}\label{ZW}
Z_C{}^c W_{AB}{}^C{}_D=Z_A{}^aZ_B{}^bZ_D{}^d W_{ab}{}^c{}_d,
\quad
Y_C W_{AB}{}^C{}_D=-Z_A{}^aZ_B{}^bZ_D{}^d C_{abd}
\end{equation}
The action of the W-tractor, as on the right hand side of \eqref{Wdef},
extends to tensor products of $\cT$ and $\cT^*$ by the Leibniz rule
and we use the shorthand $W_{AB}\sharp$ for this. For example, for any
(possibly weighted) 2-cotractor field $T_{CD}$ we have
$$
W_{AB}\sharp T_{CD}= -W_{AB}{}^E{}_C T_{ED} -W_{AB}{}^E{}_D T_{CE}.
$$

\begin{remark}
The $W$-curvature $W_{AB}{}^C{}_D$ satisfies, of course, stronger
properties if the projective structure includes the Levi-Civita
connection of a metric. An interesting case is when, in particular,
the metric is Einstein but not scalar flat, as in this case there
there is a parallel (non-degenerate) metric on the projective tractor
bundle. This can be used to raise and lower tractor indices \cite{CapGoverMacbeth14}
and it follows easily that that the $W$-curvature $W_{AB}{}^C{}_D$ has
the same algebraic symmetries as a conformal Weyl tensor. This is
potentially important for applications, but we will not exploit these
observations in the current work.
\end{remark}

\subsection{Young diagrams and some algebra}\label{alg}

\newcommand{\V}{\mathbb{V}} \newcommand{\R}{\mathbb{R}}
\newcommand{\N}{\mathbb{N}} For a real vector space $\V$ of dimension
$N$ we consider irreducible representations of $SL(\V)\cong
SL(N,\mathbb{R})$ within $\otimes^{m}\V^*$ for $m\in \mathbb{Z}_{\geq
  0}$. Up to isomorphism, these are classified by Young diagrams
\cite{Fulton97,fulton-harris} and we assume an elementary familiarity with this
notation. Each diagram is (equivalent to) a weight $(a_1,a_2,\cdots ,a_{N})$ where $m
\geq a_1\geq \ldots \geq a_{N} \geq 0$ with $\sum_{i=1}^{k}a_i=m$. We
usually omit terminal strings of 0, strictly after $a_1$, that is for
$s\geq 2$ we usually omit $a_s$ from the list if $a_s=0$.  In
particular the trivial representation of $SL(\V)$ on $\R$ (so $m=0$)
will be denoted $(0)$ rather than $(0,\cdots ,0)$ and the dual of the
defining (or fundamental) representation of $SL(\V)$ on $\V^*$ (so
$m=1$) will be denoted $(1)$ rather than $(1,0,\cdots ,0)$. Given this
notation for weights the representation space for the representation
$(a_1,\cdots,a_h)$ will usually be denoted $\V_{(a_1,\cdots ,a_h)}$,
or by the weight $(a_1,\cdots,a_h)$, simply, if $\V$ is
understood. We will term $h$ the {\em height} of the diagram.

In fact for our current purposes we shall only need the Young diagrams
of height at most 2, and $\V$ will be $\R^{n+1}$ with it standard
representation of $SL(n+1,\R)$.  The symmetric representations
$S^m\V^*$ have the diagram $(m)$, while $(k,\ell)$ with $k +\ell
=m\geq 1$, $k\geq \ell\geq 1$, can be realised by tensors $T_{B_1\ldots
  B_{k} C_1 \ldots C_\ell}$ on $\V$ which are symmetric in the
$B_i$'s, also symmetric in the $C_i$'s, and such that symmetrisation over
the first (equivalently any) $k+1$ indices
%%${\scriptstyle{B_1\ldots B_{k}}}$
vanishes: 
\begin{equation}\label{Ychar}
T_{B_1\ldots
  B_{k} C_1 \ldots C_\ell}= T_{(B_1\ldots
  B_{k}) (C_1 \ldots C_\ell)} \quad \mbox{and} \quad T_{(B_1\ldots
  B_{k} C_1)C_2 \ldots C_\ell}=0 .
\end{equation}

%% \edz{RQ: Thomas are you happy with the phrasing here or would you
%%   like the (equivalently any) added as a fully separate
%%   characterisation.}

%{\blu edz[Note to remove later: Penrose claims the characterisation here without the (equivalently
%  any). This is page 145. * To get the equivalently any -- once we know that we are in the given diagram if is clear that one cannot symmetrise over more than $k$ indices otherwise we would have a component in $\V_{k',\ell'}$ where $k'>k$.]}\\
%% [ * But it is also verified by looking at the
%%   decomposition into ireducibles of $S^k\otimes S^\ell$ using
%%   projectors that always have the first indices symmetric. This
%%   because for $T\in S^k\otimes S^k$ e.g. $T-P_{(k,k)}T$ is also
%%   manifestly in $ S^k\otimes S^k$). This equals the sum of projections
%%   of $T$ into the other components $\V_{k+1,1}$ etc. But these are all
%%   symmetric on at least the first $k+1$. Thus $T-P_{(k,k)}T$ is in
%%   $(S^k\otimes S^k)$ but symmetric on $k+1$ thus is in $S^{2k}$. Ah
%%   this argument must be wrong because so fat I'l not used any other
%%   property of $T$. So it must not be that we can write T as a sum in
%%   this ``obvious'' way]
In this article we will call these particular realisations {\em
  Young symmetries} and $\V_{(k,\ell)}$ will mean the $SL(\V)$-submodule
of $\otimes^m \V$ consisting of tensors on $\V$ with these Young
symmetries.

The key algebraic fact we need is then the following.
\begin{prop}\label{key-alg}
  The map of $SL(\V)$ representations
  \begin{equation}\label{main-map}
 \V_{(r+1)}\otimes \V_{(r)} \to \V_{(r)} \otimes  \V_{(r+1)}
\end{equation}
given by
$$
T_{B_1\ldots
  B_rB_{r+1} C_1 \ldots C_r}\mapsto T_{B_1\ldots
  B_r(B_{r+1} C_1 \ldots C_r)} 
$$
is an isomorphism.
  \end{prop}
\begin{proof}
  This is an straightforward consequence of the well known Littlewood-Richardson rules
  for decomposing the tensor product $U_{C_1\cdots C_r}\otimes
  V_{B_1\cdots B_{r+1}}\in \V_{(r)}\otimes \V_{(r+1)}$ into its direct
  sum of irreducible parts, and then the properties of these irreducibles in terms of Young symmetries as explained in \cite{Fulton97,fulton-harris,penroserindler1}.  Each of the summands is a
  representation equivalent to either $\V_{(2k+1)}$ or
  $\V_{(k,\ell)}$, with $\ell\geq 1$, $k+\ell=2r+1$, and each
  projection to such a component may be factored through the map
  \nn{main-map}.
    \end{proof}

%% \begin{remark}
%% Note that order used is important in the above. If we followed the
%% same approach to computing the $S_{B_1\cdots B_{r+1}}\otimes
%% T_{C_1\cdots C_r}$ then the irreducible summands would be each
%% symmetry over the indices ${\scriptstyle{B_1\ldots B_{r+1}}}$ but not
%%   in general over the indices ${\scriptstyle{C_1\ldots C_{r}}}$.
%%   \end{remark}
This yields the following consequence.
\begin{cor}\label{conseq}
For $r\in \mathbb{Z}_{\geq 1}$ and $k\geq \ell\geq 1$ with $k+\ell=r+1$, 
  $$
(\V_{(r+1)}\otimes \V_{(r)} )\cap (\V_{(r)}\otimes \V_{(k,\ell)})=\{0\}.
  $$
  \end{cor}
\begin{proof}
   The irreducible components of $\otimes^{r+1} \V^*$ isomorphic to $\V_{(k,\ell)}$, with $k\geq \ell\geq 1$ and
   $k+\ell=r+1$ all lie in the kernel of the map
  \begin{equation}\label{prjn}
\otimes^{r+1} \V^* \to  \V_{(r+1)} 
  \end{equation}
  However from the Proposition \ref{key-alg} the kernel of the map \nn{main-map} is trivial.
\end{proof}

In fact the kernel of \eqref{key-alg} is spanned by the irreducible components of $\otimes^{r+1} \V^*$ isomorphic to $\V_{(k,\ell)}$, with $k\geq \ell\geq 1$ and
   $k+\ell=r+1$. Thus 
it is clear that in fact the Corollary~\ref{conseq} is equivalent to the
Proposition~\ref{key-alg}. Thus it is interesting to prove this
directly. We present this here, since for our later purposes this will
be useful.

Another fact that will be useful is the following.
\begin{lemma}\label{sym-lem}
  Suppose that $T_{B_1\cdots B_rC_1\cdots C_r}=T_{(B_1\cdots B_r)(C_1\cdots C_r)}\in \V_{(r,r)}$. Then
  \begin{equation}\label{pairwise}
  T_{B_1\cdots B_rC_1\cdots C_r}=(-1)^r T_{C_1\cdots C_rB_1\cdots B_r} .
  \end{equation}
\end{lemma}
\begin{proof}
The projector  $P_{(r,r)}:\otimes^{2r}\V^*\to \V_{(r,r)}$ is given by  
\begin{equation}\label{DY}
P_{(r,r)} T= S_{(1,\ldots , r)}\circ  S_{(r+1,\ldots , 2r)}\circ S_{[1,r+1]}\circ \cdots \circ S_{[r,2r]}(T),
\end{equation}
where $S_{(1\ldots r)}$ denotes symmetrisation over  the first $r$ indices, $S_{(r+1,\ldots , 2r)}$ denotes symmetrisation over  the last $r$ indices,
 $S_{[i,j]}$ denotes  anti-symmetrisation over the two indices in, respectively, the $i^{\rm th}$ and $j^{\rm th}$ positions.

The claim in the Lemma is an immediate consequence.
  \end{proof}

In the following we extend these conventions, notations, and
definitions to vector bundles (with fibre $\V$) in the
obvious way.

\section{Killing equations: prolongation via the tractor connection} \label{main}

Here we treat the Killing type equations 
\begin{equation}\label{kill}
\nd_{(a_0}k_{a_1\cdots a_r)}=0 ,
\end{equation}
on an affine manifold with an affine connection $\nd$. For simplicity
we assume this is torsion free, but this plays almost no
role. 
There is such an equation for
each $r\in \mathbb{Z}_{>0}$ and as discussed above the equations are each projectively
invariant if we take the symmetric rank $r$ tensor to have projective
weight $2r$, i.e.  $k_{b\cdots c}\in \Gamma(\ce_{(b\cdots c)}(2r))$.
In the following,  we denote by $\cT_{(k,\ell)}$ the tractor bundle with fibre $\V_{(k,\ell)}$ where $\V=\R^{n+1}=\cT|_p$. Moreover we include the weight $w$ in the notation as $\cT_{(k,\ell)}(w)$.

Via the cotractor filtration sequence \nn{euler} we evidently have the following. 
\begin{lemma}\label{incl-lemma}
There is a projectively invariant bundle inclusion 
$$
S^r T^*M (2r)\to S^r\cT^* (r)=\cT_{(r)}(r)
$$
given by 
\begin{equation}\label{Kdef}
S^r T^*M (2r) \ni k_{b\cdots c}\mapsto K_{B\cdots C}:=\Zp{B}{b}\cdots \Zp{C}{c} k_{b\cdots c}\in \cT_{(r)}(r).
\end{equation}
\end{lemma}
\noindent Note that for $K$ as here we have
\begin{equation}\label{Xorth}
  X^BK_{B \cdots C}=0.
\end{equation}
Moreover if $K\in \cT_{(r)}(r)$ satisfies \nn{Xorth} then it is in
the image of \nn{Kdef}.

This enables a tractor interpretation of the Killing type equations,
as follows.
%%\edz{This is a partial prologation result, but the term
%%  looks strange here}
\begin{prop}\label{pp}
For each rank $r$ the equation (\ref{kill}) is equivalent to the
tractor equation
\begin{equation}\label{tkill}
\D_{(A}K_{B\cdots C)}=0,
\end{equation}
where $K_{B\cdots C}$ is given by \nn{Kdef}.
\end{prop}
\begin{proof}
  From the tractor formulae \nn{trids} and
  \nn{TD} we have
%%  $$
\begin{align*}
  \D_{A_0}K_{A_1\cdots A_r} = &\  rY_{A_0}K_{A_1A_2\cdots A_r} - Y_{A_1}K_{A_0A_2\cdots A_r}- \cdots
  -  Y_{A_r}K_{A_1A_2\cdots A_{r-1}A_0}\\& + \Zp{A_0}{a_0}\Zp{A_1}{a_1}\cdots \Zp{A_r}{a_r}\nabla_{a_0}k_{a_1\cdots a_r} ,
\end{align*}
%%$$
which implies 
\[
\D_{(A_0}K_{A_1\cdots A_r)} = \Zp{(A_0}{a_0}\Zp{A_1}{a_1}\cdots \Zp{A_r)}{a_r}\nabla_{a_0}k_{a_1\cdots a_r} ,\]
from which the result follows  immediately.
\end{proof}
In the following $K_{A_1\cdots A_r}$ will always refer to a weight $r$
symmetric  tractor as given by by \nn{Kdef}.
We now define a projectively invariant operator 
\begin{equation}\label{Ldef}
\mathcal L: S^r T^*M (2r)\to P_{(r,r)}( \otimes^{2r}\cT^*) =  \cT_{(r,r)},
\end{equation}
where $P_{(r,r)}$ is the $(r,r)$ Young symmetry as described in
expression \nn{DY}, by applying the Young projection $P_{(r,r)}$ to
$\D^{r}K$, as follows
$$
k_{c_1\cdots c_r}\mapsto P_{(r,r)}(\D_{B_1}\cdots \D_{B_r} K_{C_1\cdots C_r}),
$$
with $K_{C_1\cdots C_r}=\Zp{C_1}{c_1}\cdots \Zp{C_r}{c_r} k_{c_1\cdots c_r}$.

\begin{prop} \label{splitprop}
The operator $\mathcal L: S^r T^*M (2r)\to \cT_{(r,r)}$ of (\ref{Ldef}) is a
differential splitting operator. 
\end{prop}
\begin{proof}
We claim that 
\begin{equation}\label{recover}
 X^{B_1}\cdots X^{B_r} \Wp{C_1}{c_1}\cdots
\Wp{C_1}{c_1}P_{(r,r)}(\D_{B_1}\cdots \D_{B_r} K_{C_1\cdots C_r}) = c
k_{c_1\cdots c_r} ,
\end{equation}
where $c$ is a non-zero constant. It clearly suffices to show that 
\begin{equation}\label{recover1} X^{B_1}\cdots X^{B_r} P_{(r,r)}(\D_{B_1}\cdots \D_{B_r} K_{C_1\cdots C_r}) = c
K_{C_1\cdots C_r} .  
\end{equation}
 Contract $X^{B_1}\cdots X^{B_r} $ into the explicit
expansion of $P_{(r,r)}(\D_{B_1}\cdots \D_{B_r} K_{C_1\cdots C_r})$. 
Use (i) $[\D_A,X^B]=\delta^B_A$, (ii) $X^A \D_A f=w f$, for any tractor
field $V$ of weight $w$ (see \nn{DX}), and that (iii) $X^AK_{A\cdots C}=0$, to
eliminate all occurrences of $X$. It follows easily that the result is
$c K_{C_1\cdots C_r}$ for some constant $c$, since there is no way to
include a term involving $\D$s that has the correct valence (i.e.\ the
tractor rank $r$). That $c\neq 0$ is found by explicit computation or
more simply the fact that it is not zero in the case that the affine
connection $\nabla$ is projectively flat, as we shall see below.
\end{proof}

The above definition is motivated by the projectively flat case where
the situation is particularly elegant. (It is easily verified that the
operator $\mathcal L$ above is a co-called first BGG splitting operator, as
discussed in e.g. \cite{cgh11}, and see references therein. We will
not use this fact however.)

We conclude this section with an observation. It shows, in particular, that sections of $\cT_{(r,r)}$ that are parallel for the usual tractor connection determine solutions of \nn{kill}. These are the so-called {\em normal solutions} (see e.g. \cite{cgh11}):
\begin{prop}\label{newprop}
Let $(M,\bp)$ be a projective manifold (not necessarily flat) and 
let $L\in \Gamma(\cT_{(r,r)})$ such that
\begin{equation}\label{newprophyp}
0=
X^{B_1}\cdots X^{B_r}\D_A L_{B_1\cdots B_rC_1\cdots C_r}.
\end{equation}
Then $K_{C_1\cdots C_r}\in \Gamma(\cT_{(r)})$ defined by $K_{C_1\cdots C_r}=X^{B_1}\cdots X^{B_r}L_{B_1\cdots B_r C_1\cdots C_r}$ satisfies equation (\ref{tkill}). If we assume in addition that 
\begin{equation}\label{newprophyp1}
0=
\D_A L_{B_1B_2\cdots B_{r}C_1\cdots C_r},
\end{equation}
then 
 $L$ defines a rank $r$ Killing tensor  via (\ref{recover}) such that $L$ is a constant multiple of ${\mathcal L}(k)$.
\end{prop}
\begin{proof} The proof is a direct rewriting of (\ref{newprophyp}),
\begin{equation}
\label{newcomp}
\begin{array}{rcl}
0&=&
X^{B_1}\cdots X^{B_r}\D_{A_1} L_{B_1\cdots B_rC_1\cdots C_r}
\\
&=&
X^{B_2}\cdots X^{B_r}\left(\D_{A_1} (X^{B_1} L_{B_1,\cdots B_rC_1\cdots C_r})- L_{{A_1} B_2\cdots B_rC_1\cdots C_r} \right)
\\
&=&
-r X^{B_2}\cdots X^{B_r} L_{{A_1} B_2\cdots B_rC_1\cdots C_r} + \D_{A_1} K_{C_1\cdots
  C_r},
\end{array}
\end{equation}
where we successively apply  $[\D_A,X^B]=\delta^B_A$ to commute and  eliminate $X$'s and $\D$'s and  use  the symmetries of $L$. Note that this computation does not require any mutual commutations of $\D_A$'s. 
Now since $L_{B_2\cdots B_r(AC_1\cdots C_r)}=0$ this equation implies equation (\ref{tkill}). 
Moreover, because of the symmetries of $L$, we also have that
\[X^{C_i}K_{C_1\cdots C_r}=X^{C_i}X^{B_1}\cdots X^{B_r}L_{B_1\cdots B_rC_1\cdots C_r}=0,\]
for each $i=1,\ldots , r$. This implies that $K$ is given by a $k$ as in relation (\ref{Kdef}). 

Applying $\D_{A_r}, \ldots , \D_{A_2} $ successively to equation (\ref{newcomp}), commuting  with the $X$'s successively by  $[\D_A,X^B]=\delta^B_A$ and finally using the additional hypothesis (\ref{newprophyp1}), shows that
$\D_{A_r}\cdots \D_{A_1}K_{C_1\cdots C_r}$ is a nonzero constant multiple of $L_{{A_r}\cdots {A_1}C_1\cdots C_r}$. 
Hence, $L$ is a constant multiple of $\mathcal {\mathcal L}(k)$.
\end{proof}

\subsection{Projectively flat structures}\label{pflat}

In this subsection we restrict to affine (or projective) manifolds
that are projectively flat, i.e. where the projective tractor curvature
vanishes. According to equation~\nn{Wexp} this also means that the Thomas-$\D$ operators mutually
commute when acting on weighted tractor sections.

In the projectively flat setting we obtain a nice characterisation of Killing
tensors. 
\begin{prop}\label{key1}
Let $(M,\bp)$ be a projectively flat manifold. 
 Let $k_{c_1\cdots c_r}\in \Gamma(S^r T^*M(2r))$ and define
  $K_{C_1\cdots C_r}:= \Zp{C_1}{c_1} \cdots \Zp{C_r}{c_r}k_{c_1\cdots c_r}$, as in \nn{incl-lemma}. Then   $k$ satisfies the Killing equation (\ref{kill}) if and only if
\begin{equation}\label{result}
\D_{B_1}\cdots \D_{B_r} K_{C_1\cdots C_r}\in \Gamma(\cT_{(r,r)} ).
\end{equation}

In particular, 
on a projectively flat manifold there is a non-zero constant $c$ so that
$$
{\mathcal L}(k)= c\  \D_{B_1}\cdots \D_{B_r} K_{C_1\cdots C_r},
$$
if and only if $k$ solves (\ref{kill}).
\end{prop}

\begin{proof}
($\Rightarrow$) Since we work in the projectively flat setting the
  Thomas-$\D$ operators commute. So 
$$
\D_{B_1}\cdots \D_{B_r} K_{C_1\cdots C_r}\in \Gamma(\cT_{(r)}\otimes \cT_{(r)})
$$
%\edz{ {\blu T: I'd prefer with Gamma but we can also write somewhere that we mean also sections when writing a vector bundle} RS: Okay let's use $\Gamma$s -- we'll both need toinclude in many places}  
Suppose that (\ref{kill}) holds. Then
(\ref{tkill}) holds, so symmetrising the left hand side of the display
over any $r+1$ indices that include $C_1\cdots C_r$ results in annihilation and so we conclude
(\ref{result}) from the definition of  $\V_{(r,r)}$ and hence of $\cT_{(r,r)} $ in \nn{Ychar}.

\medskip 
($\Leftarrow$) If (\ref{result}) holds then 
$$
\D_{B_1}\cdots \D_{B_{r-1}}\D_{(B_r} K_{C_1\cdots C_r)}=0
$$ 
so  
$$
X^{B_1}\cdots X^{B_{r-1}}\D_{B_1}\cdots \D_{B_{r-1}}\D_{(B_r} K_{C_1\cdots C_r)}
= (r-1)!\  \D_{(B_r} K_{C_1\cdots C_r)}  =0,
$$
from \nn{DX}, thus we obtain the result from Proposition \ref{pp}.
\end{proof}

Here and throughout, as above, $K\in \Gamma(\cT_{(r)}(r))$ is
the image of some $k\in \Gamma(S^r T^*M(2r)) $ as in formula~\nn{Kdef}.

%\edz{RS: Removed {\blu ``In the following we assume $P_{(r,r)}$ is a projection.''} -- we could add a factor so that it is. What do you prefer? T: Adding the factor could get messy, so leave as it is.}
\begin{prop}\label{splitprop1}
The constant $c$ in equation~(\ref{recover}) is not 0. %% $r!$.
\end{prop}
\begin{proof}
  In the case that the structure is projectively flat this is
  immediate from the Proposition  \ref{key1}, 
    since $X^{B_1}\cdots X^{B_r}$
  contracted into $\D_{B_1}\cdots \D_{B_r} K_{C_1\cdots C_r}$ gives
  $r!\  K_{C_1\cdots C_r} $. But it is clear from the argument in
  the proof of Proposition \ref{splitprop} that $c$ does not depend on
  curvature, as no commutation of $\D$s is involved.
\end{proof}

\begin{thm}\label{key3} Let $(M,\bp)$ be  projectively flat manifold. Then the splitting operator $\mathcal L$ gives an isomorphism between  Killing tensors of rank $r$ and  sections of 
$\cT_{(r,r)}$ that are parallel for the 
  projective tractor connection. 
\end{thm}
\begin{proof}
Since $\mathcal L$ is a splitting operator, it does not have a kernel. Moreover, using  that $\nabla_a L=0$ is equivalent  to
$\D_A  L=0$, Proposition \ref{newprop} shows that every parallel section of $\cT_{(r,r)}$ arises as $\mathcal L(k)$ for a Killing tensor $k$.
So it remains to show that 
 $\mathcal L(k)$ is a parallel section of the projective tractor connection whenever $k$ is a Killing tensor:
Suppose that (\ref{kill}) holds. Then by Proposition~\ref{key1}, 
$$
  %%\tilde{L}(k):=
  \D_{B_1}\cdots \D_{B_r}
  K_{C_1\cdots C_r} ={\mathcal L}(k),
$$ 
and ${\mathcal L}(k)$ has weight 0 so
$$
\D_{A}{\mathcal L}(k)= Z_A{}^a\nd_a {\mathcal L}(k).
$$
Thus it suffices to show that $\D_{A}{\mathcal L}(k)=0$. But 
$$
\D_{A}{\mathcal L}(k)= \D_{A}\D_{B_1}\cdots \D_{B_r}
  K_{C_1\cdots C_r} =0,
$$ 
because of the identity $[\V_{(r+1)}\otimes \V_{(r)}]\cap [\V_{(r)}\otimes
  \V_{(r,1)}]=\{0\} $ from Corollary \ref{conseq} (where we have used \eqref{tkill} which implies that $\D K$ is a section of $\cT_{(r,1)}(2r-1)$).
\end{proof}

\medskip

As a final note in this section we observe that it is easy to
``discover'' the projectively invariant Killing equation using the
tractor machinery, as follows. Consider a symmetric rank $r$ covariant tensor 
field $k_{c_1\cdots c_r}$ of projective weight $2r$. Form 
$$ 
K_{C_1\cdots C_r}\in S^r\cT^*(r)
$$ 
by Lemma \ref{incl-lemma}. We wish to prolong this to a parallel
tractor. This requires a tractor field of weight $0$. Thus we apply
the $r$-fold composition of $\D$. Altogether we have the projectively
invariant operator
$$
k\mapsto \D_{B_1}\cdots \D_{B_r} \Zp{C_1}{c_1}\cdots \Zp{C_r}{c_r} k_{c_1\cdots c_r}= \D_{B_1}\cdots \D_{B_r} K_{C_1\cdots C_r}, 
$$
and the image has weight zero. Thus we can form 
$$
\nabla_a \D_{B_1}\cdots \D_{B_r} \Zp{C_1}{c_1}\cdots \Zp{C_r}{c_r} k_{c_1\cdots c_r},
$$
by construction it is projectively invariant and we
can ask what it means for this to be zero. 
Equivalently we seek the condition on $k$ determined by 
$$
\D_A \D_{B_1}\cdots \D_{B_r} K_{C_1\cdots C_r}=0 .
$$
But this implies $X^{C_1}\cdots X^{C_r} \D_A \D_{B_1}\cdots \D_{B_r} K_{C_1\cdots C_r}=0$ and from equation (\ref{newprophyp})
in the proof of Theorem \ref{key3} it follows that 
$$
\D_{(A} K_{B_1\cdots 
  B_r)}= 0, \qquad \text{implies} \qquad \nd_{(a}k_{b_1\cdots b_r)}=0
$$
where we again used Proposition \ref{pp}.

\subsection{Restoring curvature} \label{c-section}

We return now to the general curved case and seek the
generalisations of the results in the previous subsection. First we
observe the following first generalisation of Proposition \ref{key1}:
\begin{prop}\label{key2}
Let $k\in \Gamma(S^rT^*M(2r))$ on a general affine manifold
$(M,\nabla)$ (or projective manifold $(M,\bp)$) and $K=
K(k)\in\Gamma(\cT_{(r)} (r)) $, as in \nn{Kdef}. Then $k$ is a Killing tensor, i.e., a solution
of (\ref{kill}), if and only if we have
\begin{equation}\label{want} 
  {\mathcal L}(k)=\D_{B_1}\cdots \D_{B_r} K_{C_1\cdots C_r} +
  \operatorname{{\bf Kurv}}(K),
\end{equation}
where $\operatorname{{\bf Kurv}}$ is a specific projectively
invariant linear differential operator on $\Gamma( \cT_{(r)} (r))$, of order at most $(r-2)$, constructed with the
$W$-curvature  and the Thomas-$\D$ operators
and such that the $W$-curvature and its $\D$-derivatives appear in the
coefficients of every term.
\end{prop}
\begin{proof}
($\Rightarrow$) Suppose that $k$ solves \nn{kill}. We have 
$$
{\mathcal L}(k)= P_{(r,r)} (\D_{B_1}\cdots \D_{B_r} K_{C_1\cdots C_r} ).
$$
We expand out this expression on the right hand side using the
definition of the operator $P_{(r,r)}$ in \nn{DY}. 
We would like to
show that the resulting terms can be combined and rearranged to yield~(\ref{want}).  We have the identity (\ref{tkill}) available.  In the projectively
flat case we also have the identity $[\D_A,\D_B]=0$ as an operator on
(weighted) tractors. In the flat case the two identities are enough to
conclude (\ref{want}) (with $
\operatorname{{\bf Kurv}}(K)=0$), according to the proof of
Proposition \ref{key1}. In the curved case we perform the same formal
computation but keep track of the curvature, i.e., replace each
$[\D_A,\D_B]$ with $W_{AB}\sharp$ (instead of 0). The order statement follows by construction (or elementary weight arguments), so this proves the result in this direction and generates a specific formula for $\operatorname{{\bf Kurv}}(K)$.
% \tedz{Do we want to keep this comment? RS: I am ambivalent. Remove if you think that is better.}(This formula may depend on the way we performed the calculation, but we may fix this choice.)

($\Leftarrow$) Now  we suppose that $k \in \Gamma(S^rT^*M(2r))$ is any section such that
$$
\D_{B_1}\cdots \D_{B_r} K_{C_1\cdots C_r} +
\operatorname{{\bf Kurv}}(K)_{B_1\cdots B_rC_1\cdots C_r}
$$
 is a section of $\cT_{(r,r)}(r)$. Then
in particular 
$$
\D_{B_1}\cdots \D_{B_{r-1}}\D_{(B_r} K_{C_1\cdots C_r)} +
\operatorname{{\bf Kurv}}(K)_{B_1\cdots B_{r-1}(B_rC_1\cdots C_r)}
=0,
$$ according to \nn{Ychar}. As in the proof of Proposition \ref{key1}, we
contract now with $X^{B_1}\cdots X^{B_{r-1}}$. This contraction
annihilates the second term in the display as follows. Each of the
$X^{B_i}$'s is contracted into either a $\D_{B_i}$, into $K$, or into the curvature  $W$.
Thus every $X^{B_i}$ can be eliminated using the identities \nn{DX}, that $X^BK_{B\cdots
  C}=0$, and that similarly 
 $X^B
 %W_{BC}{}^D{}_E=0
 $ contracted into any of the lower indices of the curvature $W$ is zero.  But, by the construction of the operator
$\operatorname{{\bf Kurv}}$, in any term there are at most
$(r-2)$ $\D$ operators (either  applied to the curvature or directly to the argument) and so the identities \nn{DX} remove only $(r-2)$ of
the $(r-1)$ $X$'s. This means that in every term produced we have a
contraction of the form $X^BK_{B\cdots C}=0$, so that term vanishes,
or $X$ into $W$ so also that term vanishes.
%% and we use the first of the identities \nn{DX} to move the $X$'s
%% to the right to yield an $X^{B_i}\D_{B_i}$ or $X$ directly contracted
%% into $W$ or $K$ (without other $\D$ operators between these).  But then we have 
%% there are at most $r-2$ $\D$ operators in any term of this and
%% contraction of $X$ into $W$ is zero \edz{RS: Add this above},
%% $X^BK_{B\cdots C}=0$
Thus we are left with
$$
0=X^{B_1}\cdots X^{B_{r-1}}\D_{B_1}\cdots \D_{B_{r-1}}\D_{(B_r} K_{C_1\cdots C_r)} =(r-1)! \ \D_{(B_r} K_{C_1\cdots C_r)},
$$
as in the proof of Proposition \ref{key1}.
\end{proof}

\begin{prop} \label{preresult}
Let $k\in \Gamma(S^rT^*M(2r))$ on a general affine manifold
$(M,\nabla)$ (or projective manifold $(M,\bp)$) and $K=
K(k)\in\Gamma(\cT_{(r)} (r)) $, as in \nn{Kdef}. Then $k$ is a solution
of (\ref{kill}) if and only if we have
\begin{equation}\label{want2} 
 \D {\mathcal L}(k)=  \operatorname{{\bf Curv}}(K),
\end{equation}
where $\operatorname{{\bf Curv}}$ is a projectively invariant linear
differential operator, of order at most $(r-1)$, on
$\Gamma(\cT_{(r,r)}(r))$ given by a specific formula
constructed with the $W$-curvature, and the Thomas-$\D$ operator such
that the $W$-curvature and its derivatives appear in the coefficients
of every term. Moreover, if $\mathcal L(k)$ satisfies equation (\ref{want2}), then 
\begin{equation}\label{t}
X^{B_1}\cdots X^{B_r}\D_A\mathcal L(k)_{{B_1}\cdots _{B_r} C_1\cdots
  C_r}=0.
\end{equation}
\end{prop}
\begin{proof}
%Since $\D_A{\mathcal L}(k)=Z_A{}^a\nd_a {\mathcal L}(k)$ it suffices to study $\D_A{\mathcal L}(k)$. 

($\Rightarrow$) 
Suppose that $k$ solves (\ref{kill}).  
We apply $\D_A$ to both sides of (\ref{want}). This yields 
\begin{equation}\label{saviour}
\D_A{\mathcal L}(k)= \D_A \D_{B_1}\cdots \D_{B_r} K_{C_1\cdots C_r} + \D_A \operatorname{{\bf Kurv}}(K)_{B_1\cdots B_rC_1\cdots B_r} .
\end{equation}
In the case when $\nabla$ is projectively flat the first term on the
right can be shown to be zero by a formal calculation using just the
identities $[\D_A,\D_B]=0$ and $\D_{(A_0}K_{A_1\cdots A_r)}=0$. This
follows from the proof of Theorem \ref{key3}. Performing the same
formal calculation, but now instead replacing the commutator of $\D$'s
with $[\D_A,\D_B]=W_{AB}\sharp$ and combining the result with the
second term on the right hand side yields the result: $\D {\mathcal L}(k)$ is equal
to a specific formula for a linear differential operator
$\operatorname{\bf Curv}$ on $K$ that is constructed polynomially, and
with usual tensor operations, involving just the $W$-curvature, and
the Thomas-$\D$ operator. Thus by construction it is projectively
invariant, and also by construction (or weight arguments) the order claim follows.
%% Since $X^A\D_A {\mathcal L}(k)=0$ it
%% follows that same contraction with $X^A$ annihilates the right-hand
%% side and hence we have \nn{want2}. As a note in preparation for the
%% next part of the proof, we observe that this contraction of $X^A$
%% annihilates the right-hand side without usng that $k$ is a solution of
%% \nn{Kdef}. Thus uses that it is easily seend from the construction
%% that each term the formula the free index `$A$' is onn either a $\D_A$
%% acting on a tractor of weight zero, or on an undifferentiated
%% $W_{AB}\sharp$, each of which are annilated by contraction with $X^A$.

($\Leftarrow$) We suppose now that \nn{want2} holds with $k\in
\Gamma(S^rT^*M(2r))$,  $K$ as in \nn{Kdef} and with the operator $\operatorname{\bf Curv}$ given by the formula found the first part of the proof. So we have
$$
\D_A{\mathcal L}(k)_{B_1\cdots B_rC_1\cdots C_r}= \operatorname{\bf Curv}(K)_{AB_1\cdots B_rC_1\cdots C_r} .
$$ Note that contraction of $X^{C_1}\cdots X^{C_r}$ annihilates the
right hand side by an easy analogue of the argument used in the second
part of the proof of the Proposition \ref{key2} above: in this case
there are at most $(r-1)$ many $\D$ operators in any term but we are
contracting in $\otimes^r X$, so in each term an $X$ is contracted
directly into and undifferentiated $K$ or $W$. The result now follows
by the argument used in second part of the proof of Theorem \ref{key3}
for the projectively flat case. Thus we have just shown that we  have
the equation
(\ref{t}).
Then the result follows from the first part of Proposition \ref{newprop}.
%In the same way as in the computation (\ref{newprophyp}), which did not involve mutual commutations of $\D$'s,  
% we now eliminate the $X$'s by moving them to
%the right using the identities \nn{DX}, but not commuting any
%$\D$'s. Since $X^AK_{A\cdots C}=0$ we must finally get again an
%expression involving just one $\D$ acting on $K$. Thus we obtain again
%\nn{tkill}. Thus the result follows from Proposition \ref{pp}.
\end{proof}

For the proof of the main theorem we recall the following fact, which follows from the theory of overdetermined systems of PDE.
\begin{lemma}\label{fibrelemma}
%\tedz{New lemma  (we need actually a stronger statement)--- proof or reference?}
For every $T\in \cT_{(r,r)}|_x$, where $x\in M$, there is a local section 
$k\in \Gamma ( \mathcal S^r T^*M|_U)$, such that $T={\mathcal L}(k)|_x$.
\end{lemma}
\begin{proof}
  In the case of (projectively) flat $(M,\bp)$ this follows at once
  from the fact that in the flat case for $L\in \Gamma(T_{(r,r)})$ we have shown that
  $\nabla L=0$ implies $L={\mathcal L}(k)$. 
%  \edz{RS: check we have that. If not add. Basically this just arises a variant of the Lemma above. Just define $K=L(X,\cdots ,X)$. This is one of our usual $K$s}

  For the general case the result then follows as the formula for the
  operator ${\mathcal L}(k)$ generalises that from the flat case by the simply
  the addition (at each order) of lower order curvature terms.
%% This follows easily from the definition of the operator $L$ and that  for
%% the given weight of $k$, $D_{(A_1}\cdots D_{A_r)}K_{B_1\cdots B_r}$
%% captures the entire $r$-jet of $k$.
  \end{proof}

Now we state and prove the main results of the paper.
%
%{\blu
%I have replaced the old version of  Theorem \ref{preresultthma} and its proof by what follows. The old version is below in green for comparison. In blue are the bits  in the proof that I would like to improve and a remark that I find important.}
\begin{thm} \label{preresultthma}
Let $(M,\bp)$ be a projective manifold. Then there is a specific section $\mathcal R_A\sharp \in \cT^*M\otimes  \operatorname{End}(\cT_{(r,r)})$ 
(where we suppress the endomorphism indices) such that $X^A\mathcal R_A\sharp=0$ and such that 
the differential splitting operator $\mathcal L:\Gamma(S^rT^*M(2r)) \to \Gamma( \cT_{(r,r)}) $ gives an isomorphism between Killing tensors of rank $r$ and sections $L$ of the bundle $\cT_{(r,r)}$ that satisfy the 
 the equation 
%for every symmetric rank $r$ tensor $k\in \Gamma(S^rT^*M(2r))$  it holds the following equivalence: $k$ is a Killing tensor if and only if 
%${\mathcal L}(k)\in \Gamma( \cT_{(r,r)})$  satisfies the equation 
\begin{equation}\label{want3a} 
 \D_A L =  \mathcal R_A  \sharp L.
\end{equation}
\end{thm}

\begin{proof}
Again, the splitting operator $\mathcal L$ is injective. Hence, we have to show the following:
\begin{enumerate}
\item[(A)] For every  Killing tensor $k$ the image $\mathcal L(k)$ satisfies  equation (\ref{want3a}) with a specific  $\mathcal R_A\sharp \in \cT^*M\otimes  \operatorname{End}(\cT_{(r,r)})$ that will be determined;
  \item[(B)] $\mathcal L$ restricted to Killing tensors (i.e.\ the
    solutions of \eqref{kill}) is surjective onto the sections $L$
    that satisfy equation (\ref{want3a}), where the right hand side is
    as determined in (A).
\end{enumerate}

 We prove (A): Assume that $k$ solves \nn{kill}. Then we have equation \nn{want2},   
  $$
  \D {\mathcal L}(k)=  \operatorname{{\bf Curv}}(K).
  $$ from Proposition \ref{preresult}.  
  The operator $\operatorname{{\bf Curv}}$ is given by a formula
  polynomial  in the $W$-curvature, its $\D$ derivatives, and the
  Thomas-$\D$ operators up to order $(r-1)$. Now observe that  each term of the form
  $\D_{B_1}  \cdots \D_{B_s}K_{C_1\cdots C_r}$, for $0\le s <r$ can be replaced  
using \eqref{want} from Proposition~\ref{key2}, 
  \[
  \D_{B_1}  \cdots \D_{B_s}K_{C_1\cdots C_r}
  =
  c\  X^{s+1}\cdots X^r \mathcal L(k)_{B_1\cdots B_r C_1\cdots C_r}
  +
   \operatorname{{\bf Curv}}^{(s)}(K),
   \]
  where $ \operatorname{{\bf Curv}}^{(s)}$ is a differential operator 
  is given by a formula
  polynomial  in the $W$-curvature, its $\D$ derivatives, and the
  Thomas-$\D$ operators up to order $(s-2)$. In this way we can successively eliminate all applications of $\D$ to $K$ by terms algebraic in $\mathcal L(k)$
arriving at an equation of the form
 %\tedz{changed}
  \begin{equation}\label{want3c}
  \D_A {\mathcal L}(k)={\mathcal R}_A\sharp {\mathcal L}(k),\quad\text{with ${\mathcal R}_A\in \Gamma(\cT^* \otimes
  \operatorname{End}(\otimes^{2r}\cT^*))$,}\end{equation}
 given by a
  polynomial in the $W$-curvature and its $\D$-derivatives. 
  Now we have to verify: 
  \begin{enumerate}
  %  \item that $X^{B_1}\cdots X^{B_r} ( \mathcal R_C  \sharp {\mathcal L}(k))_{B_1 \cdots B_r C_1\cdots C_r}=0$,
  \item[(i)] that  ${\mathcal R}_A\sharp$ is indeed a section of $\cT^* \otimes
  \operatorname{End}(\cT_{(r,r)})$, and 
  \item[(ii)] that for every $L\in \cT_{(r,r)}$, the contraction of ${\mathcal R}_A\sharp L$ with $X^A$ is equal to zero.
  \end{enumerate}
%  The first point 
%  is simply a rewriting of the
%right hand side of \nn{want2} and the second part of the proof of Proposition \ref{preresult} where we showed that  right hand side of \nn{want2} is annihilated by contraction with $X^{B_1}, \ldots , X^{B_1}$.
In order to verify (i) and (ii) we have to make a key
  observation:
  Although we phrased the discussion above in a naive way that
  supposes there is a solution to \eqref{kill}, in fact to derive
  \eqref{want3c} we do not actually require that there exist
  solutions, even locally, to the equation~\eqref{kill}.  Equation
  \eqref{want3c} simply expresses relations on the jets, of a section
  $k\in \Gamma(S^rT^*M(2r))$ that are formally determined by a finite
  jet prolongation of the Killing equation \eqref{kill}. It is clear that we can  derive \eqref{want3c} at any point $x\in M$ by working with
  just the $r+1$-jet, $j^{r+1}_xk$, of $k$ at $x$. Following the argument as above, but working
  formally with such jets and assuming \eqref{kill} holds to  order $r$ at $x$, we come to
\begin{equation}\label{want3ax}
%\D_A {\mathcal L}(k) (x)={\bf Kurvalg}\sharp {\mathcal L}(k)(x) ,
\D_A L|_x=  \mathcal R_A  \sharp {\mathcal L}(k)|_x
\end{equation}
where all curvatures and their derivatives are evaluated at $x$.  From
the results in the projectively flat case we know that this is exactly
the point where the prolongation of the finite type PDE \eqref{kill}
has closed: The prolongation up to order $r$ may be viewed as simply
the introduction of new variables labelling the part of the jet that
is not constrained by the equation, and these are exactly parametrised
by the elements in the fibre $\cT_{(r,r)}(x)$. At the next order
the derivative of these variables is expressed algebraically in terms
of the variables from $\cT_{(r,r)}(x)$. That is (a key part of) the content of \eqref{want3ax}. Viewing this as a computation
in slots (via a choice of $\nabla\in \bp)$ the computation is the same
in the curved case as in the projectively flat case except that
additional curvature terms may enter when derivatives are commuted. It
follows that ${\mathcal L}(k)(x)$ may be an arbitrary element $L$ of
$\cT_{(r,r)}(x)$. Using this, and since contraction with $X^A$ annihilates the left
hand side of \eqref{want3ax} it follows that it annihilates the right hand side for any $L\in T_{r,r}(x)$. Similarly since the left hand side of \eqref{want3ax} is a section of
$(\cT^*\otimes T_{r,r})(x)$ so is the right hand side, for arbitrary $L={\mathcal L}(k)(x)$ and thus (ii) also follows.

\medskip

 Now we prove  (B): Suppose that $L\in
  \Gamma(\cT_{(r,r)})$ satisfies  \nn{want3a}  for  the specific $\mathcal R_A\in\Gamma( \cT^*\otimes \cT_{(r,r)}) $ obtained from the argument above.  
%  As in the flat case
%   we set $K_{C_1 \cdots C_r}:=X^{C_1}\cdots X^{C_r}L_{B_1,\cdots B_rC_1\cdots C_r}$. 
We now claim that 
\begin{equation}\label{want3b} 
X^{B_1}\cdots X^{B_r} ( \mathcal R_C  \sharp L)_{B_1 \cdots B_r C_1\cdots C_r}=0.
\end{equation}
Indeed, 
in the case that $L=\mathcal L(k)$ for a tensor
$k$ that solves \nn{kill}, we know from  Proposition \ref{preresult}  that $X^{C_1}\cdots X^{C_r}$ annihilates
the right hand side of equation (\ref{want3a}) for $\mathcal L(k)$, because then it is simply a rewriting of the
right hand side of \nn{want2}.
However, as mentioned above, at a point $x\in M$  and
  for $k$ satisfying \eqref{kill} to order $r$ at $x$, any element  of
  $\cT_{(r,r)}|_x$ can arise as ${\mathcal L}(k)|_x$ because this is
  the full prolonged system for the overdetermined PDE \eqref{kill}.
  %% (This last
  %% claim was proved in the projectively flat case in Section
  %% \ref{pflat}. Then by an obvious argument the prolonged system spans
  %% the same dimension the curved case.)
%% \edz{RS: We
%%   know this is true in the flat case -- beacuse then the system is
%%   equivalent to parallel transport. But anyway -- just by dimension
%%   etc it must follow from the flat case result}
Thus it follows that
$X^{C_1}\cdots X^{C_r}$ must annihilate the right hand side of
\nn{want3a} for $L$ even if $L$ is not $\mathcal L(k)$ for a  $k\in \Gamma(S^rT^*M(2r))$ 
satisfying \nn{kill}.

Having established equation~\nn{want3b}, we can apply the first part of Proposition \ref{newprop} to ensure that $L$ determines a Killing tensor $k$. 
Then we have that  $L=\mathcal L(k)$ unless the map 
\[
L_{B_1\cdots B_rC_1\cdots C_r}\mapsto K_{B_1\cdots B_r}=X^{C_1}\cdots X^{C_r}L_{B_1\cdots B_rC_1\cdots C_r}\in \mathcal T_{(r)}(r).\] has a kernel.
To exclude this possibility, assume there is a section $L $ of $\cT_{(r,r)}$ that satisfies 
(\ref{want3a})
and such that \begin{equation}\label{xL0}
X^{C_1}\cdots X^{C_r}L_{B_1\cdots B_rC_1\cdots C_r}=0.\end{equation}
The following lemma shows that this implies the vanishing of $L$.

\begin{lemma} Let $L_{B_1\cdots B_rC_1\cdots C_r}$ be a section of $\cT_{(r,r)}$ that satisfies equation (\ref{want3a}) for the specific $\mathcal R_A\sharp\in \Gamma(\cT^*\otimes \cT_{(r,r)})$. Then we have the following implication:
if 
 \begin{equation}\label{assume}
   X^{B_1}\cdots X^{B_k}L_{B_1\cdots B_k\cdots B_rC_1\cdots C_r}=0\quad\text{ for a $k\in \{1,\ldots , r\}$,}\end{equation}
%  \begin{equation}  
%$X^{B_1}\cdots X^{B_k}L_{B_1\cdots B_k\cdots B_rC_1\cdots C_r}=0$ for a $k\in \{1,\ldots , r\}$,
then 
\[ X^{B_1}\cdots X^{B_{k-1}}L_{B_1\cdots B_{k-1} \cdots B_rC_1\cdots C_r}=0,\] and hence $L_{B_1 \cdots B_rC_1\cdots C_r}=0$.
\end{lemma}
\begin{proof}
Assume that
equation (\ref{assume}) holds.
 Applying $\D_A$, the Leibniz rule for $\D_A$  gives
\begin{equation}
\label{ktok-1}
0=c \ X^{B_1}\cdots X^{B_{k-1}}L_{B_1\cdots B_{k-1}A B_{k+1}\cdots  B_rC_1\cdots C_r}+
X^{B_1}\cdots X^{B_k}\D_A
L_{B_1\cdots B_rC_1\cdots C_r},\end{equation}
with a nonzero constant $c$. Hence,  we have to show that equation (\ref{assume})
 implies 
\begin{equation}\label{hts}
X^{B_1}\cdots X^{B_k}\D_A
L_{B_1\cdots B_k\cdots B_rC_1\cdots C_r}=0,
\end{equation} by using equation (\ref{want3a}) and the specific form of $\mathcal R_A\sharp$. The proofs of the previous propositions and of~(A) provide us with the following information about $\mathcal R_A\sharp$:  In Proposition \ref{preresult} we have seen that
the expression $\operatorname{{\bf Curv}}(K)$ was of order at most $(r-1)$ in $\D$ and  is a linear combination of in terms of the form 
$\mathcal A^{(s-1)}\otimes \D^{r-s} K$ for $1\le s\le r$ and where $\mathcal A^{(s-1)}$ is a tractor of valence $s$ containing at most $s-1$ applications of $\D$ to the tractor curvature $W$. 
Then in (A) of the present proof we have expressed the terms $\D^{r-s} K$ by an $s$-fold contraction  of $\mathcal L(k)$ with  $X$. Hence $\mathcal R_A\sharp L$ is a linear combination of terms of the form 
\begin{equation}\label{form}
\mathcal A^{(s-1)}\otimes \mathcal B^{(s)},\end{equation} where
$\mathcal B^{(s)}$ is of the form $X^{E_1}\cdots X^{E_s}L_{E_1\cdots
  E_s E_{s+1}\cdots E_rC_1\ \cdots C_r}$. Because of (\ref{assume}),
the only terms that are nonzero in $\mathcal R_A\sharp L$ are those of
the form (\ref{form}) with $s<k$. Hence the terms $A^{(s-1)}$ contain
at most $(k-2)$ $\D$-derivatives of the tractor curvature. Now since
$X^AW_{AB}=0$ and therefore $X^A \D_{B}W_{AC}= - W_{BC}$,
each of the
$\mathcal A^{(s-1)}$ is annihilated by $s$ contractions with
$X$. Hence the only terms of the form (\ref{form}) that are non zero
when contracted with $k$ many $X$'s must have at least $(k+1-s)$
contractions with $X$ at $B^{(s)}$, which already is obtained by $s$
contractions with $X$. Hence the only terms $\mathcal B^{(s)} $ that
may remain nonzero when contracted with $(k+1-s)$ many $X$'s are of
the form
\begin{equation}\label{zero}
X^{B_1}\cdots X^{B_s}X^{C_1}\cdots X^{C_{k+1-s}}
L_{B_1\cdots B_s \cdots B_rC_1 \cdots C_{k+1-s}\cdots C_r}.
\end{equation}
Now an induction over $s$ shows that these terms are actually zero. 
In fact, for $s=1$ this follows from the assumpion (\ref{assume}). If $s>1$ we use that $L\in \cT_{(r,r)}$ to get 
\begin{eqnarray*}
\lefteqn{
X^{B_1}\cdots X^{B_s}X^{C_1}\cdots X^{C_{k+1-s}}
L_{B_1\cdots B_rC_1 \cdots  C_r}=}\\
&=&-\sum_{i=1}^r
X^{B_1}\cdots X^{B_s}X^{C_1}\cdots X^{C_{k+1-s}}
L_{B_1\cdots B_{s-1}C_iB_{s+1}  \cdots B_rC_1 \cdots C_{i-1}B_s  C_{i-1}\cdots C_r}\\
&=&
-(k+1-s)\ X^{B_1}\cdots X^{B_s}X^{C_1}\cdots X^{C_{k+1-s}}
L_{B_1\cdots B_rC_1 \cdots  C_r}
\end{eqnarray*}
by the induction hypothesis. This shows that the terms in (\ref{zero}) are indeed zero and finishes the proof of the lemma.
  \end{proof} 
  This shows that every $L\in \Gamma(\cT_{(r,r)})$ that satisfies equation (\ref{want3a}) is the image of a Killing tensor under the splitting operator $\mathcal L$. This finishes the proof of (B) and hence of the theorem.
   \end{proof}
 
Rewriting the result of this theorem in terms of the tractor connection gives:
\begin{cor}\label{maincor}
Let $(M,\bp)$ be a projective manifold. Then 
there is a projectively invariant section $\mathcal Q_a\sharp \in\Gamma( T^*M\otimes  \operatorname{End}(\cT_{(r,r)})$ 
such that
the splitting operator $\mathcal L$ gives an isomorphism
% for every symmetric rank $r$ tensor $k\in \Gamma(S^rT^*M(2r))$  it holds the following equivalence: $k$ is a Killing tensor if and only if 
between weighted Killing tensors of rank $r$ and sections
$L\in \Gamma( \cT_{(r,r)})$  that satisfy satisfies the equation 
\begin{equation}\label{want3t} 
 \nabla^\cT_a L=  \mathcal Q_a\sharp L,
\end{equation}
or equivalently, sections $L$ that are  parallel for connection 
\begin{equation}\label{wantconnection}
 \nabla^\cT_a-   \mathcal Q_a\sharp.\end{equation}
% on $\cT_{(r,r)}$ correspond to Killing  tensor fields of rank $r$ on $(M,\bp)$.
\end{cor}
\begin{proof}
This follows by contracting equation (\ref{want3a}) with $\Wp{A}{a}$ yielding equation (\ref{want3t}) with some $\mathcal Q_a\sharp \in\Gamma( T^*M\otimes  \operatorname{End}(\cT_{(r,r)})$.  Moreover, since 
$X^A\mathcal R_A\sharp=0$, the resulting $\mathcal Q_a$ is projectively invariant.
\end{proof}

\begin{remark}\label{int-app}
  As a final remark we note that there is a considerable gain in
understanding the prolongation of \nn{kill} in the form \eqref{want3t}
(or equivalently \eqref{wantconnection}), rather than simply as some
(possible invariant) connection $\tilde\nabla$ on $\cT_{r,r}$ without the
structure \eqref{wantconnection} (or some equivalent) made explicit. An obvious example of
such a gain is for the explicit computation of integrability
conditions. Given such a connection the standard way to compute
integrability conditions is via the curvature of $\tilde{\nabla}$,
since this must annihilate any section of $\cT_{(r,r)}$ that
corresponds to a solution of \eqref{kill}. However, because the bundle
$\cT_{(r,r)}$ has very high rank (e.g. for $r=2$ it has rank
$n^2(n^2-1)/12$) and the prolongation connection is necessarily very
complicated, computing such curvature is typically  out of reach
without the development of specialised software. However given
\eqref{want3t} we obtain integrability conditions immediately from the
curvature $\kappa$ (see \eqref{tractor_curvature}) of the normal tractor connection: Differentiating \eqref{want3t} with the latter and skewing in the obvious way we obtain 
\begin{equation}\label{int-cond}
2\nabla^{\cT}_{[b}\nabla^\cT_{a]} L= \kappa_{ba}\sharp L= \nabla_{[b} (\mathcal Q_{a]}\sharp L).
\end{equation}
Then using similar ideas to the treatments above, we can expand the
(far) right hand side by replacing any instance of $\nabla^\cT_b L$
with $Q_{b}\sharp L$ and thus, by subtracting $\kappa_{ba}\sharp L$,
obtain at once a projectively invariant 2-form with values in $\End (T_{r,r})$,
that must annihilate any $L(k)$ for $k$ solving \eqref{kill}. Thus the
existence of solutions \ref{kill} constrains the rank of this natural
projective invariant constructed from the tractor curvature and its derivatives. From there one can compute invariants that must vanish following standard ideas, as in e.g. \cite[Section 3]{gover-nurowski04} (applied there to a different problem). 
\end{remark}

\section{Explicit results for low rank} \label{examples}

\subsection{The curved  rank $r=1$ case} The rank one case is well known and here we compare it to our approach. We construct the connection corresponding to the equation
\begin{equation}\label{kv}
\nabla_{(a}k_{b)}=0 \qquad \nabla\in \bp
  \end{equation}
on $k_b\in \Gamma(T^*M(2))$ on a projective manifold $(M,\bp)$.  Following Lemma \ref{incl-lemma} we
form $K_C=\Zp{C}{c} k_c\in \cT^*(1)$, where $k_c$ is a solution of \eqref{kill}, and then according to the definition \nn{Ldef}, set
$$
{\mathcal L}(k)_{BC}: =\D_{[B}K_{C]}.
$$
Consider the case that $k$ is a
solution of \nn{kv}. Then from Proposition \ref{pp}, 
$$
\D_BK_C\in \Gamma (\Lambda^2\cT^*) ,
$$
and because the $W$-tractor satisfies the algebraic Bianchi identity
$W^{\phantom{A}}_{AB}{}^{E}{}_C+W^{\phantom{A}}_{BC}{}^{E}{}_A+W^{\phantom{A}}_{CA}{}^{E}{}_B=0$ we have
$\D_{[A}\D_{B}K_{C]}=0$, that is
$$
\D_A\D_BK_C=[\D_C,\D_B]K_A = -W^{\phantom{A}}_{CB}{}^{E}{}_AK_E .
$$
So for solutions $k$ we have
$$
\D_{A}\D_{[B}K_{C]} - W^{\phantom{A}}_{BC}{}^{E}{}_A X^F\D_{[F}K_{E]}=0.
$$
So $\nabla_a {\mathcal L}(k)_{BC}+ W_{BC}{}^{E}{}_A W^A_aX^F {\mathcal L}(k)_{EF}=0$. 
But for any $k\in \Gamma(T^*M(2))$
$$
X^F\D_{[F}K_{E]} =X^F {\mathcal L}(k)_{FE}= K_E .
$$
Thus the projectively invariant connection on $\Lambda^2\cT^*$ is given by
$$
\nabla_a V_{BC}+ W^{\phantom{A}}_{BC}{}^{E}{}_A \Wp{A}{a}X^F V_{EF} .
$$
It is easily checked that this agrees with the formula \nn{eg} from
the introduction (and so that connection $\overline{\nabla}$ is
projectively invariant).

\subsection{The curved   rank $r=2$ case}\label{curvedrk1sec}
Here we consider the case $r=2$. We will make the computations in Section \ref{c-section} explicit and in particular provide explicit formulae for the curvature tractor fields fields $\mathcal R_A\sharp$ and $\mathcal Q_a\sharp$.

The first observation was established as part of a more involved argument in the second part of the proof of Proposition \ref{preresult}:
\begin{lemma}\label{obs1}
If $K_{DE}\in \Gamma(\cT_{(2)} (2)) $, then 
\[X^DX^E\D_A\D_B\D_CK_{DE}=6\D_{(A}K_{BC)}.\]
In particular, $ X^EX^D\D_A\D_B\D_CK_{DE}$ is totally symmetric.
\end{lemma}
\begin{proof} A direct computation using the relation (\ref{[DX]})  implies
\begin{equation}\label{crucial}
X^C\D_A V_{CB\cdots} =\left[X^C,\D_A\right]V_{CB\cdots}+ \D_A(X^CV_{CB\cdots})
=
-
V_{AB\cdots}+ \D_A(X^CV_{CB\cdots}).\end{equation}
This can be used to commute  $X^E$ and $X^D$ past the $\D$'s until $X^EK_{EA}=0$ can be applied.
\end{proof}
Now we study the projection $P:=P_{(2,2)} $ from $\otimes^4\cT ^*$ to $\cT_{(2,2)}$ defined in (\ref{DY}).
If $S_{BCDE}$ is an element in $\otimes^4\cT^*$ that is symmetric in $D$ and $E$, i.e., $S_{BCDE}=S_{BC(DE)}$, then a straightforward computation shows that for $S_{BCDE}\in \otimes^2\cT ^* \otimes \cT_{(2)}$ we have
\begin{equation}\label{proj}
(PS)_{BCDE}
= \tfrac{1}{4}\left(S_{(BC)DE}+ S_{(DE)BC}\right)
-\tfrac{1}{8}\left( S_{(DC)BE} + S_{(EB)CD} + S_{(DB)CE}+ S_{(EC)BD}\right).
\end{equation}
This implies ideed that
\[(S_{(ijk)}PS)_{BCDE}=0,
\]
i.e., the symmetrisation of $PS$ over any three indices $1\le i<j<k \le 4$ vanishes. 

Next, for a section $K_{DE}\in \Gamma (\cT_{(2)} (2) )$ we set
 $S_{BCDE}:=\D_B\D_CK_{DE}$. Note that the differential splittig operator $\mathcal L$ is given by
 $\mathcal L(k)_{bd}=(P\D^2K)_{BCDE}$.
  % \edz{RS: Thomas throughout, as here, you are mixing (in your notation) sections of tractor bundles with elements of $\V_{(k,l)}$ etc . Perhaps we should say e.g.
% $K_{DE}\in P_{(2)}(\otimes^2\cT^*)(r)$? T: see above.\\ {\color{blue} This Lemma is also part of the proof of $\Leftarrow$ of Propn \ref{preresultthma}}}
We obtain the following statement, which was already observed in the proof Theorem \ref{key3} and Proposition \ref{preresultthma} for general rank:
\begin{lemma}\label{obs2}
If $K_{DE}\in  \Gamma (\cT_{(2)} (2) )$, then 
\[
X^EX^D\D_A(P\D^2K)_{BCDE}=
\tfrac{1}{4}
X^EX^D\D_A\D_B\D_CK_{DE}
\]
\end{lemma}
\begin{proof}
We use the formula (\ref{proj}) for $S_{BCDE}:=\D_B\D_CK_{DE}$ and apply $\D_A$ to it. Using relation~(\ref{crucial}) as well as $X^DK_{DB}=0$ and equations (\ref{DX}), a direct computation shows that each of the last eight terms in the right hand side of (\ref{proj}) vanishes when contracted with $X^D$ and $X^E$.  For example,
\begin{eqnarray*}
X^EX^D\D_A\D_D\D_CK_{BE}
&=& -\D_CK_{BA}+X^D\D_A\D_CK_{BD}-X^D\D_A\D_CK_{BE}
\\
&=&
-\D_AX^D\D_CK_{BD}-X^D\D_A\D_DK_{BC}
\\
&=&
2 \D_AK_{BC}-\D_AX^D\D_DK_{BC}
\\
&=&
0.
\end{eqnarray*}
A similar computation shows that
\[
X^EX^D\D_A\D_D\D_EK_{BC}
=
\D_AK_{BC} +[X^E,\D_A]\D_EK_{BC} =0.\]
Hence, equation (\ref{proj}) implies that 
\[
X^EX^D\D_A(P\D^2K)_{BCDE}=
\tfrac{1}{4}
X^EX^D\D_A\D_{(B}\D_C)K_{DE}=
\tfrac{1}{4}
X^EX^D\D_A\D_B\D_CK_{DE},
\]
where the second equality follows from Lemma \ref{obs1}.
\end{proof}

The following lemma will give a formula for the projection $P$, when restricted to $\cT\otimes \cT_{(2,1)}$,
%\edz{RS: Adjusted here -- but again this is the mixed notation}
i.e., applied to
 $S_{BCDC}\in \cT^*\otimes \cT_{(2,1)}$.
 
 \begin{lemma}\label{projDDKlem}
Let  $P:=P_{(2,2)} $ be the projection of $\otimes^4\cT^*$ onto $\cT_{(2,2)}$ defined above
and $S_{BCDC}\in \cT^*\otimes \cT_{(2,1)}$. Then
\begin{equation}\label{projDDK}
(PS)_{BCDE}
%&=& \tfrac{3}{4} S_{(BC)DE}-\tfrac{3}{8}\Big( S_{[DC]BE}+ S_{[EB]CD}+ S_{[DB]CE}+ S_{[EC]BD}\Big)\\[2mm]
=\tfrac{3}{4}\left( S_{BCDE}
- S_{[BC]DE}\right)
-\tfrac{3}{8}\left( S_{[DC]BE}+ S_{[EB]CD}+S_{[DB]CE}+ S_{[EC]BD}\right).
\end{equation}
\end{lemma}
\begin{proof}
We use  equation (\ref{proj}) under the additional assumption that $S_{BCDC}\in \V^*\otimes \V_{(2,1)}$, i.e.,
\begin{equation}\label{killing}
 S_{B(CDE)}=0.\end{equation}
 For the  the third term on the right-hand-side in  (\ref{proj})  we compute
\[
S_{(DC)BE}
\ =\ 
%\tfrac{1}{2}\left( S_{CDBE}+  S_{DCBE}\right)
%\\
 S_{CDBE}+ S_{[DC]BE}
\ =\ -  S_{CBDE}- S_{CEBD} +S_{[DC]BE},
\]
where the last equation uses equation (\ref{killing}). 
This allows to compute the sum of the last four terms in (\ref{proj}) as 
\begin{equation}\label{terms3456}
\begin{array}{rcl}
\lefteqn{S_{(DC)BE} + S_{(EB)CD} + S_{(DB)CE}+ S_{(EC)BD}=}\\
&=&
-4 S_{(CB)DE} -S_{CEDB}- S_{CDBE}  -S_{BEDC}-S_{BDEC}
\\
&&{} + S_{[DC]BE}+S_{[EB]CD}+S_{[DB]CE}+S_{[EC]BD}
\\
&=&
-2 S_{(CB)DE} + S_{[DC]BE}+S_{[EB]CD}+S_{[DB]CE}+S_{[EC]BD},
\end{array}
\end{equation}
where the last equation again follows from (\ref{killing}).

Now we look at the second term on the right-hand-side of (\ref{proj}): using (\ref{killing}) we get that
\[
\begin{array}{rcl}
S_{(DE)BC}
&=&
-\tfrac{1}{2}
\left(
 S_{DBCE}+
 S_{DCEB}+
 S_{EBCD}+
 S_{ECDB}\right)
\\[2mm]
&=&
-\tfrac{1}{2}
\left(
 S_{BDCE}+
 S_{CDEB}+
 S_{BECD}+
 S_{CEDB}\right)
\\[1mm]
&&{}-
\left(
 S_{[DB]CE}+
 S_{[DC]EB}+
 S_{[EB]CD}+
 S_{[EC]DB}\right)
\\[2mm]
&=& S_{(BC)DE}-
\big(
 S_{[DB]CE}+
 S_{[DC]EB}+
 S_{[EB]CD}+
 S_{[EC]DB}\big).
\end{array}
\]
Hence, equation (\ref{pairwise}) from the flat case generalises to 
\begin{equation}\label{pairwise1}
 S_{(BC)DE}=
 S_{(DE)BC}
+
\big(
 S_{[DB]CE}+
 S_{[DC]EB}+
 S_{[EB]CD}+
 S_{[EC]DB}\big).
\end{equation}
Then putting (\ref{terms3456}) and (\ref{pairwise1}) together, for $S_{BCDE}\in \V^*\otimes \V_{(2,1)}$, finishes the proof.
\end{proof}

Now assume that 
$\D_C$ is the Thomas $\D$-operator and $K_{DE}$ is symmetric such that 
\begin{equation}\label{tkilling}
\D_{(C}K_{DE)}=0.\end{equation}
Then set $S_{BCDE}:=\D_B \D_CK_{DE}$ in the above equations. Observe that 
\begin{eqnarray*}
 S_{[BC]DE}&=&\D_{[B}\D_{C]}K_{DE}\ =\ 
\tfrac{1}{2}\left(
\D_{B}\D_{C}K_{DE}-\D_{C}\D_{B}K_{DE}\right)
\\
&=&
\tfrac{1}{2} W_{BC}\ \sharp K_{DE}
\ =\
-W_{BC}{}^F{}_{(D}K_{E)F}.
\end{eqnarray*}
Then, from Lemma \ref{projDDKlem} we get an explicit version of the curvature terms in Proposition \ref{key2}:
\begin{prop}\label{LKprop}
Let $\D$ be the Thomas $\D$-operator for a projective structure with curvature $W_{AB}{}^C{}_D$ and let $P$ be the projection from $\cT^*\otimes \cT_{(2,1)}$ to $\cT_{(2,2)}$. 
Then $K\in \Gamma(\cT^*_{(2)})$ 
satisfies $\D_{(A}K_{BC)}=0$, i.e.,  $\D_{A}K_{BC}\in \cT_{(2,1)}$, if and only if  
\begin{equation}\label{LK}
(P\D^2K)_{BCDE}
=\tfrac{3}{4} \D_B\D_CK_{DE}
-\tfrac{3}{8}\left(
W_{BC}\sharp K_{DE}
+
W_{D(B}\sharp K_{C)E}
+
W_{E(B} \sharp K_{C)D}\right)
\end{equation}
that is
\[
\D_B\D_CK_{DE}
+\tfrac{1}{2}\left(
W_{BC}\sharp K_{DE}
+
W_{D(B}\sharp K_{C)E}
+
W_{E(B} \sharp K_{C)D}\right)\in V_{(2,2)}.
\]
\end{prop}
\begin{proof}
One direction immediately follows from Lemma \ref{projDDKlem} applied to $S_{BCDE}:=\D_B\D_CK_{ED}$. 

For the other direction assume that equation (\ref{LK}) holds. Contracting with $X^B$ and noting that $X^BW_{B\cdots}=0$ as well as $X^BK_{BC}=0$ implies that
\begin{equation}
\label{DKsplit}
X^B(P\D^2K)_{BCDE}
=
\tfrac{3}{4} 
X^B
\D_B\D_CK_{DE}
=
\tfrac{3}{4} 
\D_CK_{DE}
\end{equation}
from the definition of $\D_B$. Hence, since $P\D^2K\in \Gamma(\cT_{(2,2)})$, 
the symmetrisation over $CDE$ vanishes.
\end{proof}

%
%\edz{{\blu
%Check Proposition \ref{splitprop} and Proposition \ref{splitprop1} which state that
%\[X^BX^C(P\D^2K)_{BCDE}=2 K_{DE}.
%\]}\\ RS: I changed those}
Note that, from equation (\ref{LK}) we obtain that
\begin{equation}
\label{Ksplit}
X^BX^C(P\D^2K)_{BCDE}
=\tfrac{3}{4} 
X^BX^C
\D_B\D_CK_{DE}
=
\tfrac{3}{4} 
X^C
\D_CK_{DE}
=
\tfrac{3}{2} K_{DE},\end{equation}
because of (\ref{DX}) and (\ref{XW0}).

%
%{\blue
%... We try to simplify this using the Bianchi identity,
%\[
%0=
%W_{BC}{}^F{}_{D}
%+
%W_{CD}{}^F{}_{B}
%+
%W_{DB}{}^F{}_{C},
%\]
%which gives
%\[
%W_{[AB}\sharp K_{C]D} = W_{[AB}{}^F{}_{|D|}K_{C]F}
%,
%\]
%where $[AB|D|C]$ means that $\D$ is excluded from the skew symmetrisation over $A$,$B$ and $C$.
%
%
%Unfortunately, this does not help a lot
%\[
%\begin{array}{rcl}
%\tfrac{4}{3}(P\D^2K)_{BCDE}
%&=&2 \D_B\D_CK_{DE}\
%\\&&
%+
%2 
%W_{BC}{}^F{}_{(D}K_{E)F}
%+
%W_{DC}{}^F{}_{(B}K_{E)F}
%+ W_{EB}{}^F{}_{(C}K_{D)F}
%+
%W_{DB}{}^F{}_{(C}K_{E)F}
%+
%W_{EC}{}^F{}_{(B}K_{D)F}
%\\
%&=&2 \D_B\D_CK_{DE}\
%\\&&
%+
%\left(
%2 
%W_{BC}{}^F{}_{D}
%-
%W_{CD}{}^F{}_{B}
%+
%W_{DB}{}^F{}_{C}
%   \right) K_{EF}
%\\
%&&
%+
%\left(
%2 
%W_{BC}{}^F{}_{E} 
%+ W_{EB}{}^F{}_{C}
%-
%W_{CE}{}^F{}_{B}
% \right)K_{DF}
%\\
%&&
%+
%\left( W_{DC}{}^F{}_{E} 
%-
%W_{CE}{}^F{}_{D}
%\right) K_{BF}
%\\
%&&+
%\left(
% W_{EB}{}^F{}_{D}
%-
%W_{DE}{}^F{}_{B}
% \right)K_{CF}
% \\
%&=&
%2 \D_B\D_CK_{DE}\
%\\&&
%+
%\left(
% W_{BC}{}^F{}_{D}
%-2
%W_{CD}{}^F{}_{B}
%   \right) K_{EF}
%\\
%&&
%+
%\left(
% W_{BC}{}^F{}_{E} 
%-
%2 W_{CE}{}^F{}_{B}
% \right)K_{DF}
%\\
%&&
%-
%\left( W_{ED}{}^F{}_{C} 
%+
%2
%W_{CE}{}^F{}_{D}
%\right) K_{BF}
%\\
%&&+
%\left(
%2 W_{EB}{}^F{}_{D}
%+
%W_{BD}{}^F{}_{E}
% \right)K_{CF}
% \\
% &=&2 \D_B\D_CK_{DE}+
% 4 K_{F[C} W_{E]B}{}^F{}_{D}
% -2
% K_{F[B} W_{E]D}{}^F{}_{C}
% -
% 4
% W_{CD}{}^F{}_{(B} K_{E)F}
%\end{array}
%\] 
%because of the wrong signs  ....}
%
%\bigskip
Next we determine the connection for which $(P\D^2K)_{BCDE}$ is going to be parallel, i.e., we determine explicitly the curvature terms in Proposition~\ref{preresult}, Theorem~\ref{preresultthma} and Corollary \ref{maincor}.  To get a formula for its covariant derivative with respect to the projective tractor connection, we apply $\D$ to the equality in   Proposition \ref{LKprop}  to get
%
%\[
%(PT)_{BCDE}
%=\tfrac{3}{8}\left( 2T_{BCDE}
%-2 T_{[BC]DE}
%- T_{[DC]BE}- T_{[EB]CD}-T_{[DB]CE}- T_{[EC]BD}\right).
%\]
\begin{equation}\label{dlk}
4\D_C(P\D^2K)_{DEAB}
=3 \D_C\D_D\D_EK_{AB}
-\tfrac{3}{2}\D_C\left(
W_{DE}\sharp K_{AB}
+
W_{A(D}\sharp K_{E)B}
+ 
W_{B(D}\sharp K_{E)A}\right).
\end{equation}
We are now going to obtain a formula for $T_{CDEAB}= \D_C \D_D\D_E K_{AB}\in \otimes^5\V^*$. This is achieved by the following lemmas.

\begin{lemma}\label{observe}
For every $T\in \otimes^5\cT^*$ it holds
\begin{eqnarray*}
\lefteqn{
T_{C(DE)AB}+T_{D(EC)AB}+T_{E(CD)AB}
 }\\
&=&3T_{CDEAB}
+3T_{C[ED]AB}
+T_{D[EC]AB}
+T_{E[DC]AB}
+2T_{[EC]DAB}
+2T_{[DC]EAB}.
\end{eqnarray*}
\end{lemma}
\begin{proof}
The poof  is by inspection.\end{proof}
\begin{lemma}\label{D3Klemma}
Let  $T_{ABCDE}\in\otimes^2 \cT^*\otimes \cT_{(2,1)}$, i.e., $T_{AB(CDE)}=0$.
Then 
\begin{eqnarray*}
-3T_{CDEAB}
&=&
2T_{[EC]DAB}
+2T_{[DC]EAB}
+2T_{[AC]BDE}
+2T_{[AD]BEC}
+2T_{[AE]BCD}
\\&&
+2T_{A[BC]DE}
+2T_{A[BD]EC}
+2T_{A[BE]CD}
\\
&&
+3T_{C[ED]AB}
+
 T_{C[DA]BE}+
 T_{C[DB]EA}+
 T_{C[EA]BD}+
 T_{C[EB]DA}
+T_{C[AB]DE}
\\&&
+T_{D[EC]AB}
+
 T_{D[EA]BC}+
 T_{D[EB]CA}+
 T_{D[CA]BE}+
 T_{D[CB]EA}
+T_{D[AB]EC}
\\&&
+T_{E[DC]AB}
+
 T_{E[CA]BD}+
 T_{E[CB]DA}+
 T_{E[DA]BC}+
 T_{E[DB]CA}
+T_{E[AB]CD}.
\end{eqnarray*}
\end{lemma}

\begin{proof}
First we  can swap the pair  $AB$ with $DE$ by using (\ref{pairwise1}) for the second equality in
\begin{eqnarray*}
T_{ABCDE}
&=&
T_{C(AB)DE}+T_{C[AB]DE}+2T_{[AC]BDE}+2T_{A[BC]DE}
\\
&=&
T_{C(DE)AB}
+
 T_{C[DA]BE}+
 T_{C[DB]EA}+
 T_{C[EA]BD}+
 T_{C[EB]DA}
 \\
 &&
+T_{C[AB]DE}+2T_{[AC]BDE}+2T_{A[BC]DE}.
\end{eqnarray*}
In an analogous computation as in  the flat case, this can be used to evaluate
\begin{align*}
0=&3T_{AB(CDE)}
\\
=&
T_{C(DE)AB}+T_{D(EC)AB}+T_{E(CD)AB}
\\&
+
 T_{C[DA]BE}+
 T_{C[DB]EA}+
 T_{C[EA]BD}+
 T_{C[EB]DA}
+T_{C[AB]DE}+2T_{[AC]BDE}+2T_{A[BC]DE}
\\&
+
 T_{D[EA]BC}+
 T_{D[EB]CA}+
 T_{D[CA]BE}+
 T_{D[CB]EA}
+T_{D[AB]EC}+2T_{[AD]BEC}+2T_{A[BD]EC}
\\&
+
 T_{E[CA]BD}+
 T_{E[CB]DA}+
 T_{E[DA]BC}+
 T_{E[DB]CA}
+T_{E[AB]CD}+2T_{[AE]BCD}+2T_{A[BE]CD}
\end{align*}
Now we apply Lemma \ref{observe} to the terms $T_{C(DE)AB}+T_{D(EC)AB}+T_{E(CD)AB}$ in this equation to get
 to get
\begin{align*}
0
=&
3T_{CDEAB}
+3T_{C[ED]AB}
+T_{D[EC]AB}
+T_{E[DC]AB}
+2T_{[EC]DAB}
+2T_{[DC]EAB}
\\&
+
 T_{C[DA]BE}+
 T_{C[DB]EA}+
 T_{C[EA]BD}+
 T_{C[EB]DA}
+T_{C[AB]DE}+2T_{[AC]BDE}+2T_{A[BC]DE}
\\&
+
 T_{D[EA]BC}+
 T_{D[EB]CA}+
 T_{D[CA]BE}+
 T_{D[CB]EA}
+T_{D[AB]EC}+2T_{[AD]BEC}+2T_{A[BD]EC}
\\&
+
 T_{E[CA]BD}+
 T_{E[CB]DA}+
 T_{E[DA]BC}+
 T_{E[DB]CA}
+T_{E[AB]CD}+2T_{[AE]BCD}+2T_{A[BE]CD},
%\\
%&=&
%3T_{CDEAB}
%\\
%&&
%+2T_{[EC]DAB}
%+2T_{[DC]EAB}
%+2T_{[AC]BDE}
%+2T_{[AD]BEC}
%+2T_{[AE]BCD}
%\\&&
%+2T_{A[BC]DE}
%+2T_{A[BD]EC}
%+2T_{A[BE]CD}
%\\
%&&
%+3T_{C[ED]AB}
%+
% T_{C[DA]BE}+
% T_{C[DB]EA}+
% T_{C[EA]BD}+
% T_{C[EB]DA}
%+T_{C[AB]DE}
%\\&&
%+T_{D[EC]AB}
%+
% T_{D[EA]BC}+
% T_{D[EB]CA}+
% T_{D[CA]BE}+
% T_{D[CB]EA}
%+T_{D[AB]EC}
%\\&&
%+T_{E[DC]AB}
%+
% T_{E[CA]BD}+
% T_{E[CB]DA}+
% T_{E[DA]BC}+
% T_{E[DB]CA}
%+T_{E[AB]CD}.
\end{align*}
which implies the formula in the lemma.
\end{proof}
By applying this lemma to 
$T_{CDEAB}= \D_C \D_D\D_E K_{AB}\in\Gamma( \otimes^2\cT^*\otimes \cT_{(2,1)})$ for $K_{AB}\in \Gamma(\cT_{(2)})$
and by replacing  skew-symmetrisations by curvature, for example,
\[
T_{[EC]DAB}
=
\tfrac{1}{2}\left(\D_E\D_C\D_DK_{AB}-\D_C\D_E\D_DK_{AB}\right)
=
\tfrac{1}{2}W_{EC}\sharp 
\D_DK_{AB}
\]
and 
\[
T_{A[BC]DE}
=\tfrac{1}{2}\left(
\D_A\D_B\D_CK_{DE}-
\D_A\D_C\D_BK_{DE}\right)
=
\tfrac{1}{2}
\D_A (W_{BC}\sharp K_{DE}),\]
we obtain the following result. Here and henceforth we use the following convention:
the notation $|B|$ or $| A\cdots B|$ means that the index $B$, or the indices $A\cdots B$, are excluded from any surrounding symmetrisation.

\begin{prop}\label{DLKprop}
Let $\D$ be the Thomas $\D$-operator for a projective structure with curvature $W_{AB}{}^C{}_D$ and let $P$ be the  map from $\cT^*\otimes \cT_{(2,1)}$ to $\cT_{(2,2)}$ defined in (\ref{DY}). 
%\edz{RS: It is not really a projection. Shall we say we'll use the term loosely?\\ RQ: Fix notation? T: done.} 
Then $K\in \cT^*_{(2)}$ satisfies $\D_{(A}K_{BC)}=0$, i.e.,  $\D_{A}K_{BC}\in \cT_{(2,1)}$, if and only if, 
\begin{equation}\label{DLK}
\begin{array}{rcl}
 \D_C(P\D^2K)_{DEAB}
&=&
\tfrac{1}{2}W_{C(D}\sharp \D_{E)}K_{AB}
-\tfrac{3}{4} W_{A(C}\sharp \D_{|B|}K_{DE)}
-\tfrac{3}{4} \D_A (W_{B(C}\sharp K_{DE)})
\\
&&
-\tfrac{1}{8}
\D_{C}\left(W_{AB}\sharp K_{DE}
- W_{E(A}\sharp K_{B)D}-W_{D(A}\sharp K_{B)E}\right)
\\&&
-\tfrac{1}{8}\D_D
\left(
W_{AB}\sharp K_{EC}
+
W_{EC}\sharp K_{AB}
+ 
2 W_{E(A}\sharp K_{B)C}+
2
 W_{C(A}\sharp K_{B)E}\right)
\\&&
-\tfrac{1}{8}\D_E
\left(
W_{AB}\sharp K_{DC}
+W_{DC}\sharp K_{AB}
+
2 W_{C(A}\sharp K_{B)D}+
2 W_{D(A}\sharp K_{B)C}\right).
\end{array}
\end{equation}
\end{prop}
\begin{proof}
First assume that 
 equation (\ref{DLK}) holds. We contract this equation with $X^A$ and $X^B$. It is a direct computation to see the then the right hand side is zero: to see this,  recall that  $X^AW_{A\cdots}=0$ and  $X^AK_{AC}=0$ and that equation (\ref{crucial}) applied to $V_{C\cdots}$ with 
 $X^CV_{C\cdots}=0$ gives
\begin{equation}\label{crucial1}
X^C\D_A V_{CB\cdots} =\left[X^C,\D_A\right]V_{CB\cdots}+ \D_A(X^CV_{CB\cdots}) =- V_{AB\cdots}.\end{equation}
Then, from the obtained $X^AX^B\D_C(P\D^2K)_{DEAB}=0$ and from Lemmas \ref{obs1} and \ref{obs2}
we obtain the required symmetry of $\D_CK_{ED}$.

For the other direction we apply Lemma \ref{D3Klemma} to
 $T_{CDEAB}= \D_C \D_D\D_E K_{AB}\in \otimes^2\cT^*\otimes \cT_{(2,1)}$. 
Equation in Lemma \ref{D3Klemma}  then becomes
\begin{eqnarray*}
-3\D_C\D_D\D_EK_{AB}
&=&
-2 W_{C(E}\sharp 
\D_{D)}K_{AB}
+ 3 W_{A(C}\sharp \D_{|B|}K_{DE)}
+3\D_A (W_{B(C}\sharp K_{DE)})
\\
&&
+
\tfrac{1}{2}\D_{C}\left(3 W_{ED}\sharp K_{AB}+W_{AB}\sharp K_{DE}
- 2 W_{A(E}\sharp K_{D)B}-2W_{B(D}\sharp K_{E)A}\right)
\\
&&
+\tfrac{1}{2}\D_D
\left(
W_{AB}\sharp K_{EC}
+
W_{EC}\sharp K_{AB}
+ 
2 W_{E(A}\sharp K_{B)C}+
2
 W_{C(A}\sharp K_{B)E}\right)
\\&&
+\tfrac{1}{2}\D_E
\left(
W_{AB}\sharp K_{DC}
+W_{DC}\sharp K_{AB}
+
2 W_{C(A}\sharp K_{B)D}+
2 W_{D(A}\sharp K_{B)C}\right).
\end{eqnarray*}
Now we plug this in for the term $\D_C\D_D\D_EK_{AB}$ in  (\ref{DLK}) that was obtained by differentiating the equality in \ref{LKprop}:
\begin{eqnarray*}
4\D_C(P\D^2K)_{BCDE}&
=&
3 \D_C\D_D\D_EK_{AB}
-\tfrac{3}{2}\D_C\left(
W_{DE}\sharp K_{AB}
+
W_{A(D}\sharp K_{E)B}
+ 
W_{B(D}\sharp K_{E)A}\right)
\\
&=&
%2 W_{C(E}\sharp 
%\D_{D)}K_{AB}
%- 3 W_{A(C}\sharp \D_{|B|}K_{DE)}
%-3\D_A (W_{B(C}\sharp K_{DE)})
%\\
%&&
%-
%\tfrac{1}{2}\D_{C}\left(3 W_{ED}\sharp K_{AB}+W_{AB}\sharp K_{DE}
%- 2 W_{A(E}\sharp K_{D)B}-2W_{B(D}\sharp K_{E)A}\right)
%\\
%&&
%-\tfrac{1}{2}\D_D
%\left(
%W_{AB}\sharp K_{EC}
%+
%W_{EC}\sharp K_{AB}
%+ 
%2 W_{E(A}\sharp K_{B)C}+
%2
% W_{C(A}\sharp K_{B)E}\right)
%\\&&
%-\tfrac{1}{2}\D_E
%\left(
%W_{AB}\sharp K_{DC}
%+W_{DC}\sharp K_{AB}
%+
%2 W_{C(A}\sharp K_{B)D}+
%2 W_{D(A}\sharp K_{B)C}\right)
%\\&&
%-\tfrac{3}{2}\D_C\left(
%W_{DE}\sharp K_{AB}
%+
%W_{A(D}\sharp K_{E)B}
%+ 
%W_{B(D}\sharp K_{E)A}\right)
%\\
%&=&
2 W_{C(E}\sharp 
\D_{D)}K_{AB}
- 3 W_{A(C}\sharp \D_{|B|}K_{DE)}
-3\D_A (W_{B(C}\sharp K_{DE)})
\\
&&
-
\tfrac{1}{2}\D_{C}\left(W_{AB}\sharp K_{DE}
+ W_{A(E}\sharp K_{D)B}+ W_{B(D}\sharp K_{E)A}\right)
\\
&&
-\tfrac{1}{2}\D_D
\left(
W_{AB}\sharp K_{EC}
+
W_{EC}\sharp K_{AB}
+ 
2 W_{E(A}\sharp K_{B)C}+
2
 W_{C(A}\sharp K_{B)E}\right)
\\&&
-\tfrac{1}{2}\D_E
\left(
W_{AB}\sharp K_{DC}
+W_{DC}\sharp K_{AB}
+
2 W_{C(A}\sharp K_{B)D}+
2 W_{D(A}\sharp K_{B)C}\right).
\end{eqnarray*}
This finishes the proof.
\end{proof}

\newcommand{\W}[3]{W_{#1}{}^{#2}{}_{#3}}
\newcommand{\DK}[2]{\D_{#1}K_{#2}}
Now are going to expand the terms in (\ref{DLK}) using the Leibniz rule
\begin{equation}\label{DKleibniz}
\D_A (W_{BC}\sharp K_{DE})
=
(\D_A W_{BC})\sharp K_{DE} +W_{BC}\sharp (\D_A K_{DE)})+ \W{BC}{H}{A} \DK{H}{DE},
\end{equation}
and then substituting $K_{DE}$ and $\D_AK_{DE}$ terms by contractions of $X^F$ with $L_{FADE}=(P\D^2K)_{FADE}$ using relations (\ref{Ksplit}) and (\ref{DKsplit}):
\[
K_{DE}=\tfrac{2}{3}X^FX^GL_{FGDE},\qquad \D_AK_{DE}=\tfrac{4}{3}X^FL_{FADE}.
\]
To this end, first one checks that  $X^FW_{FBCD}=0$ and $\D_AX^F=\delta_A{}^F$ imply that
\[
W_{BC}\sharp(X^FQ_{F\cdots})=
X^F W_{BC}\sharp Q_{F\cdots} ,\]
and \[
\D_A W_{BC}\sharp(X^FQ_{F\cdots})=
X^F \D_AW_{BC}\sharp Q_{F\cdots} -\W{BC}{H}{A}Q_{H\cdots},
\]
for any tensor $Q_{F\cdots}$. For $Q=L$ and $Q=X^FL_{F\cdots}$ this implies
\[
W_{BC}\sharp (\D_A K_{DE)})
=
\tfrac{4}{3}W_{BC}\sharp (X^FL_{FADE})
=
\tfrac{4}{3}
X^F W_{BC}\sharp L_{FADE}\]
and 
\begin{eqnarray*}
(\D_A W_{BC})\sharp K_{DE}
&=&
\tfrac{2}{3}
(\D_A W_{BC})\sharp (X^FX^GL_{FGDE})
\\
&=&
\tfrac{2}{3}
X^FX^G\D_A W_{BC}\sharp L_{FGDE}
-
\tfrac{4}{3}
X^F\W{BC}{H}{A}L_{FHDE}
.\end{eqnarray*}
Substituting this into equation (\ref{DKleibniz}), the terms $\W{BC}{H}{A} \DK{H}{DE}$ are  cancelled 
%\edz{RQ: should we admit surprise?}
and we get
\begin{equation}\label{DLleibniz}
\D_A (W_{BC}\sharp K_{DE})
=
\tfrac{2}{3}
X^FX^G\D_A W_{BC}\sharp L_{FGDE}
+ 
\tfrac{4}{3}
X^F W_{BC}\sharp L_{FADE}.
\end{equation}
%\edz{RS: The notation $\D_A W_{BC}\sharp$ might seem ambiguous. Let's discuss.
%\\
%{\blu T: Why is it more ambiguous than $W_{BC}\sharp$?
%\\
%The bracket on the lhs makes clear what is meant, doesn't it?}
%}
Then we compute step by step the terms in the right-hand-side of (\ref{DLK}):
\begin{eqnarray*}
\lefteqn{
W_{C(D}\sharp \D_{E)}K_{AB}
-\tfrac{1}{4}\left(
\D_D
\left(
W_{EC}\sharp K_{AB}
\right)
+\D_E
\left(
W_{DC}\sharp K_{AB}
\right)\right)=}\\
&=&
2 X^F W_{C(D}\sharp L_{E)FAB}
-\tfrac{1}{3}X^FX^G \D_{(D}W_{E)C}\sharp L_{FGAB}.
\end{eqnarray*}
Next we consider the terms that are not evidently symmetric in $A$ and $B$: using $L_{A(CDE)}=0$ as well as the second Bianchi identity for the Weyl tensor we compute
\begin{eqnarray*}
\lefteqn{
-\tfrac{3}{4}\left( \D_A (W_{B(C}\sharp K_{DE)}) + W_{A(C}\sharp \D_{|B|}K_{DE)}\right)  }\\
\lefteqn{\qquad
-\tfrac{1}{8}\left(
\D_{C}\left(W_{AB}\sharp K_{DE}\right)
+
\D_{D}\left(W_{AB}\sharp K_{EC}\right)
+
\D_{E}\left(W_{AB}\sharp K_{CD}\right)
\right)=}
\\
\qquad&=&
-2 X^F W_{(B|(C}\sharp L_{DE)|A)F}-\tfrac{1}{2} X^FX^G\D_{(A}W_{B)(C}\sharp L_{DE)FG}\hspace{6cm}
\\
&=&
\tfrac{2}{3}
X^F\left( W_{C(A}\sharp L_{B)FED}
-
2
W_{(A|(D}\sharp L_{E)C|B)F}
\right)
\\
&&{}
- \tfrac{1}{6}
X^FX^G \left(
\D_{(A}W_{B)C}\sharp L_{DEFG}
+2
\D_{(A}W_{B)(D}\sharp L_{E)CFG}\right)
%\\
%&=&
%-\tfrac{3}{2}
%\D_{(A}W_{B)(C}\sharp K_{DE)}
%-\tfrac{3}{2}
%W_{B(C}\sharp \D_{|A|}K_{DE)}
%-\tfrac{3}{2}
%\W{(C|(A}{H}{B)}\DK{H|}{DE)}
%-\tfrac{3}{2} W_{A(C}\sharp \D_{|B|}K_{DE)}
%\\
%&=&
%-\tfrac{3}{2}
%\D_{(A}W_{B)(C}\sharp K_{DE)}
%+3
%W_{(C|(A}\sharp \D_{B)|}K_{DE)}
%-\tfrac{3}{2}
%\W{(C|(A}{H}{B)}\DK{H|}{DE)}
,
\end{eqnarray*}
and \begin{eqnarray*}
\lefteqn{
\tfrac{1}{2}
\D_{C}\left( W_{E(A}\sharp K_{B)D}+W_{D(A}\sharp K_{B)E}\right)}\\
\lefteqn{\qquad
-\D_D
\left(
 W_{E(A}\sharp K_{B)C}+
 W_{C(A}\sharp K_{B)E}\right)
 -\D_E
\left( W_{C(A}\sharp K_{B)D}+
 W_{D(A}\sharp K_{B)C}\right)=
}
\\
&=&
\tfrac{4}{3}X^F\left(
W_{C(A}\sharp L_{B)FED}
+
2W_{(D|(A}\sharp L_{B)F|E)C}+ 3W_{(D|(A}\sharp L_{B)|E)FC}\right)
\\
&&
-\tfrac{2}{3}X^FX^G\left( 
2\D_{(D}W_{E)(A}\sharp L_{B)CFG}
+
\D_{(A}W_{|C(D}\sharp L_{E)|B)FG} + \D_{(D}W_{|C(A}\sharp L_{B)|E)FG}\right).\hspace{2cm}
%\\
%&=&
%-\D_{(D}W_{E)(A}\sharp K_{B)C}
%-\tfrac{1}{2}\D_{(A}W_{|C(D}\sharp K_{E)|B)} -\tfrac{1}{2}\D_{(D}W_{|C(A}\sharp K_{B)|E)}
%\\
%&&
%+\tfrac{3}{2}W_{(A|(D}\sharp \DK{E)|}{B)C}{\blu +\tfrac{1}{2}W_{C(A}\sharp \DK{B)}{DE}} -\tfrac{1}{2}W_{(D|(A}\sharp \DK{B)|}{E)C}
%\\
%&&
%+\W{(A|(D}{H}{E)} \DK{H|}{B)C} 
%-\tfrac{1}{2}\W{C(D}{H}{|(A}\DK{|H|}{B)|E)} -\tfrac{1}{2}\W{C(A}{H}{|(D}\DK{|H|}{E)|B)}
%%&=&
%%\W{(A|(D}{H}{E)} \DK{H|}{B)C} 
%%+\tfrac{3}{2}W_{(A|(D}\sharp \DK{E)|}{B)C}-\tfrac{1}{2}W_{(D|(A}\sharp \DK{B)|}{E)C}
%%-\tfrac{1}{2}W_{C(A}\sharp \DK{B)}{DE}
%%\\
%%&&-\tfrac{1}{2}\W{C(A}{H}{|(D}\DK{|H|}{E)|B)}
%%-\tfrac{1}{2}\W{C(D}{H}{|(A}\DK{|H|}{B)|E)}
%%\\
%%&&
%%-\D_{(D}W_{E)(A}\sharp K_{A)C}
%%-\tfrac{1}{2}\D_{(A}W_{|C(D}\sharp K_{E)|B)} -\tfrac{1}{2}\D_{(D}W_{|C(A}\sharp K_{B)|E)}
\end{eqnarray*}
Now note that because of the pairwise symmetry of $L$ and the skew symmetry of $W$, we have
\[ W_{(A|(D}\sharp L_{E)C|B)F}=-W_{(D|(A}\sharp L_{B)F|E)C}.\]
This allows to collect some of the terms above as
\begin{eqnarray*}
\lefteqn{ \tfrac{2}{3}W_{(D|(A}\sharp L_{B)F|E)C}
+W_{(D|(A}\sharp L_{B)|E)FC}
-
\tfrac{4}{3}
W_{(A|(D}\sharp L_{E)C|B)F}=}
\\
&=&
2 W_{(D|(A}\sharp L_{B)F|E)C}
+
W_{(D|(A}\sharp L_{B)|E)FC}
\\
&=&
 W_{(D|(A}\sharp L_{B)F|E)C}
+
W_{(A|(D}\sharp L_{E)F|B)C},
\end{eqnarray*}
where the last equality follows from 
$L_{ECBF}=L_{BFEC}$ and 
 $L_{B(FEC)}=0$. Hence, we  we get the following formula for $ \D_CL_{DEAB}$ for $L:=P(\D^2K)$:
\begin{equation}\label{DLexpanded}
\begin{array}{rcl}
 \D_CL_{DEAB}
&=&
X^F \left(   W_{C(D}\sharp L_{E)FAB}
 +  
W_{C(A}\sharp L_{B)FED}\right)
\\[2mm]&&
+X^F \left(
W_{(D|(A}\sharp L_{B)F|E)C}+ W_{(A|(D}\sharp L_{E)F|B)C}
\right)
\\[2mm]
&&
- \tfrac{1}{6}
X^FX^G \left(
\D_{(D}W_{E)C}\sharp L_{ABFG}+
\D_{(A}W_{B)C}\sharp L_{DEFG}\right)
\\[2mm]
&&
-\tfrac{1}{3}X^FX^G\left( 
\D_{(D}W_{E)(A}\sharp L_{B)CFG}
+
\D_{(A}W_{B)(D}\sharp L_{E)CFG}\right)
\\[2mm]
&&
-\tfrac{1}{6}X^FX^G\left( 
\D_{(A}W_{|C(D}\sharp L_{E)|B)FG} + \D_{(D}W_{|C(A}\sharp L_{B)|E)FG}\right).
%\\[2mm]
%&&
%xxx
%\\[2mm]
%&+&
%\W{(A|(D}{H}{E)} \DK{H|}{B)C} 
%-\tfrac{1}{2}\W{C(D}{H}{|(A}\DK{|H|}{B)|E)} -\tfrac{1}{2}\W{C(A}{H}{|(D}\DK{|H|}{E)|B)}
%\\[2mm]
%&&-\tfrac{3}{2}
%\W{(C|(A}{H}{B)}\DK{H|}{DE)}
%+\tfrac{1}{2}\W{C(D}{H}{E)}\DK{H}{AB}
%\\[2mm]
%&&
%+\tfrac{3}{2}W_{(A|(D}\sharp \DK{E)|}{B)C}{\blu +\tfrac{1}{2}W_{C(A}\sharp \DK{B)}{DE}} -\tfrac{1}{2}W_{(D|(A}\sharp \DK{B)|}{E)C}
%\\[2mm]
%&&
%+3
%W_{(C|(A}\sharp \D_{B)|}K_{DE)}
%+
%\tfrac{3}{2}W_{C(D}\sharp \D_{E)}K_{AB}
%\\[2mm]
%&&
%-\D_{(D}W_{E)(A}\sharp K_{B)C}
%-\tfrac{1}{2}\D_{(A}W_{|C(D}\sharp K_{E)|B)} -\tfrac{1}{2}\D_{(D}W_{|C(A}\sharp K_{B)|E)}
%\\[2mm]
%&&
%-\tfrac{3}{2}
%\D_{(A}W_{B)(C}\sharp K_{DE)}
%-\tfrac{1}{2}\D_{(D}W_{E)C}\sharp K_{AB}
%%xxx
%%\\[2mm]
%%&=&
%%\W{(A|(D}{H}{E)} \DK{H|}{B)C} 
%%+\tfrac{3}{2}W_{(A|(D}\sharp \DK{E)|}{B)C}-\tfrac{1}{2}W_{(D|(A}\sharp \DK{B)|}{E)C}
%%-\tfrac{1}{2}W_{C(A}\sharp \DK{B)}{DE}
%%\\[2mm]
%%&&-\tfrac{1}{2}\W{C(A}{H}{|(D}\DK{|H|}{E)|B)}
%%-\tfrac{1}{2}\W{C(D}{H}{|(A}\DK{|H|}{B)|E)}
%%\\[2mm]
%%&&
%%-\tfrac{3}{2}
%%W_{B(C}\D_{|A|}K_{DE)}
%%-\tfrac{3}{2} W_{A(C}\sharp \D_{|B|}K_{DE)}
%%-\tfrac{3}{2}
%%\W{(C|(A}{H}{B)}\DK{H|}{DE)},
%%\\[2mm]
%%&&
%%+\tfrac{3}{2}W_{C(D}\sharp \D_{E)}K_{AB}
%%+\tfrac{1}{2}\W{C(D}{H}{E)}\DK{H}{AB}
%%\\[2mm]
%%&&
%%-\D_{(D}W_{E)(A}\sharp K_{A)C}
%%-\tfrac{1}{2}\D_{(A}W_{|C(D}\sharp K_{E)|B)} -\tfrac{1}{2}\D_{(D}W_{|C(A}\sharp K_{B)|E)}
%%\\[2mm]
%%&&
%%-\tfrac{3}{2}
%%\D_{(A}W_{B)(C}\sharp K_{DE)} -\tfrac{1}{2}\D_{(D}W_{E)C}\sharp K_{AB}
\end{array}
\end{equation}
%{\blu We note that each term appears twice with pairwise $AB$ and $DE$ symmetry.}
Having this formula, we can formulate the following result:
\begin{thm} 
\label{DLprop}
Let $(M,\bp)$ be an arbitrary projective manifold. Then the splitting operator $\mathcal L: S^2T^*M(4)\to \cT_{(2,2)}$ gives an isomorphism between weighted Killing tensors of rank $2$ and 
 sections
$L_{DEAB}$ of the tractor bundle $\mathcal T_{(2,2)}$ of weight zero
 that satisfy equation 
 (\ref{DLexpanded}).
\end{thm}
\begin{proof} Given a rank $2$ tensor $k_{ab}$ we  define $L_{DEAB}=\D_D\D_EK_{AB}$ and $L_{DEAB}:=(P\D^2K)_{DEAB}$. Then, if $k_{ab} $ is Killing,  it follows from Proposition \ref{DLKprop} and the above computations that $L_{DEAB}$ satisfies equation (\ref{DLexpanded}). 

On the other hand, let $L_{DEAB}$ be a section of $\mathcal T_{(2,2)}$ of weight zero that satisfies equation (\ref{DLexpanded}). 
Contracting (\ref{DLexpanded}) with $X^D$ and $X^E$, one can easily check, using the same arguments as before and that $L_{(DEF)B}=0$, that the right-hand-side vanishes and thus
\[
0=
X^DX^E\D_CL_{DEAB}
%=-2X^DL_{DCAB}+\D_C( X^DX^EL_{DEAB})=-2X^DL_{DCAB}+\D_C K_{AB}.
\]
Then 
from Proposition \ref{newprop} it follows that $L_{DEAB}$ defines a Killing tensor $k_{ab}$. Moreover we see that $L_{DEAB}=\mathcal L(k)_{DEAB}$ unless the map 
\[
L_{DEAB}\mapsto K_{AB}=X^DX^EL_{DEAB}\in \mathcal T_{(2)}(2).\] has a kernel.
So lets assume there is a section $L_{DEAB} $ of $\cT_{(2,2)}$ that satisfies 
(\ref{DLexpanded})
and such that \begin{equation}\label{xxL0}
X^DX^EL_{DEAB}=0.\end{equation}
 Applying $\D_C$ to  this  and using $0=
X^DX^E\D_CL_{DEAB}$ implies that $0=X^DL_{DEAB}$. Applying $\D_C$ to this 
 gives
\[
0=L_{CEAB}+X^D\D_CL_{DEAB}= L_{CEAB}.\]
Here the second equality uses (\ref{DLexpanded}), which allows us to compute
\begin{eqnarray*}
X^D\D_CL_{DEAB}
&=&
X^DX^F\left( W_{C(A}\sharp L_{B)FED}
+\tfrac{1}{2}
W_{E(A}\sharp L_{B)FDC}\right).
\end{eqnarray*}
But now $L_{B(FED)}=0 $ and (\ref{xxL0}) imply that
\[
X^DX^F W_{C(A}\sharp L_{B)FED}=-X^DX^F W_{C(A}\sharp L_{B)DEF}=0,\]
 which proves that $X^D\D_CL_{DEAB}=0$ and  finishes the proof.
\end{proof}

%
%\hrule
%
%Now we are using (\ref{Ksplit}) and (\ref{DKsplit}) to express terms with $K_{AB}$ and with $\D_AK_{BC}$ by  terms contractions of $X^A$ with $L:=P\D^2K$. Here the indices $H$, $F$ and $G$ are excluded from any symmetrisation.
%\newcommand{\LDKX}[2]{L_{F#1 #2}X^F}
%\newcommand{\LDK}[2]{L_{F#1 #2}}
%\begin{equation}\label{DLexpanded}
%\begin{array}{rcl}
%\lefteqn{
%2 \D_CL_{DEAB}=}\\
%&=&
%X^F\W{(A|(D}{H}{E)} \LDK{H|}{B)C} 
%-\tfrac{1}{2}X^F\W{C(D}{H}{|(A}\LDK{|H|}{B)|E)}
%-\tfrac{1}{2}X^F\W{C(A}{H}{|(D}\LDK{|H|}{E)|B)}
%\\
%&&
%-\tfrac{3}{2}X^F
%\W{(C|(A}{H}{B)}\LDK{H|}{DE)}
%-\tfrac{1}{2}X^F\W{C(D}{H}{E)}\LDK{H}{AB}
%\\
%&&
%+\tfrac{3}{2}X^FW_{(A|(D}\sharp \LDK{E)|}{B)C}
%+
%\tfrac{1}{2}X^FW_{C(A}\sharp \LDK{B)}{DE}
%-\tfrac{1}{2}X^FW_{(D|(A}\sharp \LDK{B)|}{E)C}
%\\
%&&
%+3
%X^FW_{(C|(A}L_{|F|B)|DE)}
%+\tfrac{3}{2}X^FW_{C(D}\sharp L_{F|E)AB}
%\\
%&&
%-X^FX^G\D_{(D}W_{E)(A}\sharp L_{|FG|B)C} 
%-\tfrac{1}{2} X^FX^G \D_{(A}W_{|C(D}\sharp L_{|FG|E)|B)}
% -\tfrac{1}{2}X^FX^G \D_{(D}W_{|C(A}\sharp L_{|FG|B)|E)}
%\\
%&&
%-\tfrac{3}{2}
%X^FX^G\D_{(A}W_{B)(C}\sharp L_{|FG|DE)} -\tfrac{1}{2}X^FX^G\D_{(D}W_{E)C}\sharp L_{|FG|AB}
%\end{array}
%\end{equation}

%\edz{RS: Part of the material of this remark below is discussed implicilty in the Proof of Theorem \ref{preresultthma}.}
%\begin{remark}\label{inv-remark}
Note that the right hand side of (\ref{DLexpanded}) indeed defines a section $\mathcal R_C\sharp $ of $\cT^*\otimes \cT_{(2,2)}$ as claimed in the proof of Theorem \ref{preresultthma}.

In order to extract a covariant derivative from this, we have to
contract it with $\Wp{C}{c}$. In general this contraction is not
projectively invariant.  However, since $L_{DEAB}$ has weight zero,
applying $\D_C$ to it and contracting with $X^C$ gives zero, $X^C \D_C
L_{DEAB}=0$.  Hence, the contraction $\Wp{C}{c} \D_C L_{DEAB}$ is also
projectively invariant for sections $L_{DEAB}$ that satisfy equation~(\ref{DLexpanded}). However we need that the curvature term in   right hand side of  (\ref{DLexpanded}) 
is projectively invariant as claimed in the proof of Theorem \ref{preresultthma}, i.e., that the right hand side of  (\ref{DLexpanded}) is projectively invariant for {\em any} $L_{DEAB}\in \cT_{(2,2)}$  not only for solutions of (\ref{DLexpanded}).
This is the statement of the following lemma.
%\tedz{new}
\begin{lemma}
\label{inv-remark}
For  any $L_{ABDE}\in \cT_{(2,2)}$
the right hand side in
equation (\ref{DLexpanded}) gives zero when contracted with
$X^C$. In particular, the section of $\cT^*\otimes \mathrm{End}(\cT_{(2,2)})$ defined by the right hand side in
  (\ref{DLexpanded}) is projectively invariant.
\end{lemma}
\begin{proof}
%This gives a way of detecting which terms of the right hand
%side when contracted with $W^C_{c}$ are projectively invariant, by
%checking which terms vanish when contracted with $X^C$: 
Clearly both
of the terms of the form $X^CW_{C(D}\sharp L_{E)FAB}$ in the first
line of (\ref{DLexpanded}) vanish separately because $X^CW_{C ABC}=0$.
Also both terms of the form $X^C X^F X^G \D_{(D}W_{E)(A}\sharp
L_{B)CFG}$ in the fourth line of (\ref{DLexpanded}) vanish separately
because $L_{B(CFG)}=0$.  Similarly both terms of the form $X^CX^FX^G
\D_{(D}W_{E)C}\sharp L_{ABFG}$ in the third line of (\ref{DLexpanded})
vanish separately because $X^CW_{C ABC}=0$ and
\[ X^C
\D_{(D}W_{E)C}\sharp L_{ABFG}
=
-
\delta_{(D}^C W_{E)C}\sharp L_{ABFG}
=
 W_{(EC)}\sharp L_{ABFG}
=0.\]
All the other terms in the second and fifth line 
of (\ref{DLexpanded})
do not vanish separately but cancel against each other when contracted with $X^C$. In fact we have
\[
X^C\left( 
\D_{(A}W_{|C(D}\sharp L_{E)|B)FG} + \D_{(D}W_{|C(A}\sharp L_{B)|E)FG}\right)
=-
W_{(A|(D}\sharp L_{E)|B)FG}-W_{(D|(A}\sharp L_{B)|E)FG}=0,
\]
and for the terms in the second line 
 \[
 X^C X^F \left( 
W_{(D|(A}\sharp L_{B)F|E)C}+ W_{(A|(D}\sharp L_{E)F|B)C}
\right)=0,
\]
 because of the skew-symmetry of $W_{DA}$.
 \end{proof}

 In order to obtain from equation (\ref{DLexpanded}) an equation involving the tractor derivative $\nabla_c$, we have to contract it with $\Wp{C}{c}$.   First we look at terms that for which the contracted index $C$ is at the curvature (or its derivative) $W_{AC}$. These will turn out to be manifestly invariant as we can eliminate $\Wp{C}{c}$:
First we observe that
\[
\Wp{C}{c} W_{CA}\sharp L_{BFED} = \Zp{A}{a}\kappa_{ca}\sharp L_{BFED},\]
where $\kappa_{ca}{}^H{}_G$ is the tractor curvature defined in (\ref{tractor_curvature}).
 Hence, for the terms in the first line in equation~(\ref{DLexpanded}) we get
 \[
 X^F \left(   W_{C(D}\sharp L_{E)FAB}
 +  
W_{C(A}\sharp L_{B)FED}\right)
=
X^F\left(
\Zp{(A}{a}\kappa_{|ca|}\sharp L_{B)FED}
+
\Zp{(D}{a}\kappa_{|ca|}\sharp L_{E)FAB}
\right),
\]
which is manifestly invariant. Next we compute, using formulae (\ref{trids})  
 and that the weight of $\W{CD}{H}{F}$ is $-2$, that
\begin{eqnarray*}
\Wp{C}{c}\D_AW_{BC}&=&
-2 Y_A\Zp{B}{b}\kappa_{bc} + \Zp{A}{a} \Wp{C}{c}\nabla_aW_{BC}
\\
&=&
-2 Y_A\Zp{B}{b}\kappa_{bc} + \Zp{A}{a}\left(\nabla_a  (\Wp{C}{c}W_{BC})-\nabla_a\Wp{C}{c} W_{BC}\right)
\\
&=&
-2 Y_A\Zp{B}{b}\kappa_{bc} + \Zp{A}{a}\nabla_a  (\Zp{B}{b}\kappa_{bc}) 
\\
&=&
-(2 Y_A\Zp{B}{b}+Y_B \Zp{A}{b})\kappa_{bc} + \Zp{A}{a}\Zp{B}{b}\nabla_a  \kappa_{bc},
\end{eqnarray*}
because $\nabla_a\Wp{C}{c} W_{BC} =-P_{ac}X^CW_{BC} =0$ and $\nabla_a  \Zp{B}{b}=-\delta_{a}^bY_B$.
Hence, for the expressions in the third line of (\ref{DLexpanded}) we get, 
\[
X^FX^G 
\D_{(A}W_{B)C}\sharp L_{DEFG}
=
X^FX^G 
\left(-3Y_{(A}\Zp{B)}{b}\kappa_{bc}+ \Zp{(A}{a}\Zp{B)}{b}\nabla_a  \kappa_{bc}\right)\sharp  L_{DEFG}
\]
and 
\[
X^FX^G 
\D_{(D}W_{E)C}\sharp L_{ABFG}
=
X^FX^G 
\left(-3Y_{(D}\Zp{E)}{b}\kappa_{bc} + \Zp{(D}{a}\Zp{E)}{b}\nabla_{a}  \kappa_{bc}\right)\sharp  L_{ABFG}.
\]
Similarly we get for the expressions in the fifth line of (\ref{DLexpanded}),
\begin{eqnarray*}
\lefteqn{\left( 
\D_{(A}W_{|C(D}\sharp L_{E)|B)FG} + \D_{(D}W_{|C(A}\sharp L_{B)|E)FG}\right)
=}
\\
&=&
-\tfrac{1}{2}\left( 
\D_{(A}W_{D)C}\sharp L_{BEFG} + \D_{(A}W_{E)C}\sharp L_{DBFG}
+
\D_{(B}W_{D)C}\sharp L_{AEFG} + \D_{(B}W_{E)C}\sharp L_{DAFG}
\right)
\\
&=&
\tfrac{1}{2}\left(3Y_{(A}\Zp{D)}{b}\kappa_{bc}- \Zp{(A}{a}\Zp{D)}{b}\nabla_a  \kappa_{bc}\right)\sharp  L_{BEFG}
\\
&&
+\tfrac{1}{2}\left(3Y_{(A}\Zp{E)}{b}\kappa_{bc}- \Zp{(A}{a}\Zp{E)}{b}\nabla_a  \kappa_{bc}\right)\sharp  L_{BDFG}
\\
&&
+\tfrac{1}{2}\left(3Y_{(B}\Zp{D)}{b}\kappa_{bc}- \Zp{(B}{a}\Zp{D)}{b}\nabla_a  \kappa_{bc}\right)\sharp  L_{AEFG}
\\
&&
+\tfrac{1}{2}\left(3Y_{(B}\Zp{E)}{b}\kappa_{bc}- \Zp{(B}{a}\Zp{E)}{b}\nabla_a  \kappa_{bc}\right)\sharp  L_{ADFG}.
\end{eqnarray*}
Finally, we compute
\[
W_{(D|(A}\sharp L_{B)F|E)C}+ W_{(A|(D}\sharp L_{E)F|B)C}
=
\Zp{(A}{a}\Zp{|(D}{d}\kappa_{|ad|}\sharp \left( L_{E)F|B)C}
- L_{E)C|B)F} 
\right)
\]
and 
\[
\D_{(A}W_{B)(D}\sharp L_{E)CFG}
=
-\left(
3 
Y_{(A}\Zp{B)}{b} \Zp{(D}{d}\kappa_{|bd|}
-
\Zp{(A}{a}\Zp{B)}{b} \Zp{(D}{d}\nabla_{|a}\kappa_{bd|}
\right)\sharp L_{E)CFG},
\]
to
rewrite equation (\ref{DLexpanded}) in terms of the tractor connection as
\begin{equation}\label{DLexpandedtractor}
\begin{array}{rcl}
 \nabla_cL_{DEAB}
&=&
X^F\left(
\Zp{(A}{a}\kappa_{|ca|}\sharp L_{B)FED}
+
\Zp{(D}{a}\kappa_{|ca|}\sharp L_{E)FAB}
\right)
\\[1mm]
&&
+
X^F \Wp{C}{c}
\Zp{(A}{a}\Zp{|(D}{d}\kappa_{|ad|}\sharp \left( L_{E)F|B)C}
- L_{E)C|B)F} 
\right)
\\[1mm]
&&
-
\tfrac{1}{12}X^FX^G \left(3Y_{(A}\Zp{D)}{b}\kappa_{bc}- \Zp{(A}{a}\Zp{D)}{b}\nabla_a  \kappa_{bc}\right)\sharp  L_{BEFG}
\\[1mm]
&&
-
\tfrac{1}{12}X^FX^G \left(3Y_{(A}\Zp{E)}{b}\kappa_{bc}- \Zp{(A}{a}\Zp{E)}{b}\nabla_a  \kappa_{bc}\right)\sharp  L_{BDFG}
\\[1mm]
&&
-
\tfrac{1}{12}X^FX^G \left(3Y_{(B}\Zp{D)}{b}\kappa_{bc}- \Zp{(B}{a}\Zp{D)}{b}\nabla_a  \kappa_{bc}\right)\sharp  L_{AEFG}
\\[1mm]
&&
-
\tfrac{1}{12}X^FX^G \left(3Y_{(B}\Zp{E)}{b}\kappa_{bc}- \Zp{(B}{a}\Zp{E)}{b}\nabla_a  \kappa_{bc}\right)\sharp  L_{ADFG}
\\[1mm]
&&
+
 \tfrac{1}{6}
X^FX^G 
\left(3Y_{(A}\Zp{B)}{b}\kappa_{bc}- \Zp{(A}{a}\Zp{B)}{b}\nabla_a  \kappa_{bc}\right)\sharp  L_{DEFG}
\\[1mm]&&
+
 \tfrac{1}{6}
X^FX^G 
\left(3Y_{(D}\Zp{E)}{b}\kappa_{bc} - \Zp{(D}{a}\Zp{E)}{b}\nabla_{a}  \kappa_{bc}\right)\sharp  L_{ABFG}
%-
%\Wp{C}{c}X^F \left(
% W_{(A|(D}\sharp L_{E)F|B)C}
%  - W_{(A|(D}\sharp L_{E)C|B)F} 
%\right)
%\\[1mm]
%&&
\\[1mm]
&&
-\tfrac{1}{3} X^FX^G \Wp{C}{c}\left(
3 
Y_{(A}\Zp{B)}{b} \Zp{(D}{d}\kappa_{|bd|}
-
\Zp{(A}{a}\Zp{B)}{b} \Zp{(D}{d}\nabla_{|a}\kappa_{bd|}
\right)\sharp L_{E)CFG}
\\[1mm]
&&
-\tfrac{1}{3}X^FX^G\Wp{C}{c}\left( 
3
Y_{(D}\Zp{E)}{b} \Zp{(A}{d}\kappa_{|bd|}
-
\Zp{(D}{a}\Zp{E)}{b} \Zp{(A}{d}\nabla_{|a}\kappa_{bd|}
\right)\sharp L_{B)CFG},
%
%+\tfrac{1}{3}\Wp{C}{c} X^FX^G\left( 
%\D_{(D}W_{E)(A}\sharp L_{B)CFG}
%+
%\D_{(A}W_{B)(D}\sharp L_{E)CFG}\right)
\end{array}
\end{equation}
where $\nabla_c$ is the projective tractor connection and $\kappa_{bc}$ its curvature.
The right hand side of this equation defines the section $\mathcal Q_a\sharp\in \Gamma(T^*M\otimes \cT_{(2,2)})$ in Corollary \ref{maincor}.
%{\blu  I would like to improve the second and the last two lines by getting rid of $\Wp{C}{c}$ somehow, to make evident that it is invariant (we know that from Lemma \ref{inv-remark}). But I don't know how to simplify terms like  $X^F\Wp{C}{c}L_{AFCD}$ ... 
%One could use 
%\begin{eqnarray*}
%X^F\Wp{C}{c}L_{AFCD}&=&\tfrac{1}{2}\Wp{C}{c}\D_{A}K_{CD}
%\\
%&=&\Wp{C}{c}Y_AK_{CD}+\tfrac{1}{2}\Wp{C}{c}\Zp{A}{a}\nabla_aK_{BC}
%\\
%&=&Y_A\Zp{D}{d}K_{cd}+\tfrac{1}{2}\Zp{A}{a}\nabla_aK_{Bc}
%\\
%&=&Y_A\Zp{D}{d}K_{cd}+\tfrac{1}{2}\Zp{A}{a}(\Zp{B}{b} \nabla_aK_{bc} -Y_BK_{ac} ),
%\end{eqnarray*}
%but this uses $k_{bc}$ and hence $\Wp{C}{c}$ ...
%}
Hence we arrive at:

\begin{thm} Let $(M,\bp)$ be an arbitrary projective manifold. Then the splitting operator $\mathcal L: S^2T^*M(4)\to \cT_{(2,2)}$ gives an isomorphism between weighted Killing tensors of rank $2$  and 
 sections
$L_{DEAB}$ of the tractor bundle $\mathcal T_{(2,2)}$ of weight zero
 that satisfy equation (\ref{DLexpandedtractor})
for the projective tractor connection $\nabla_a$, 
or equivalently,  parallel sections of the connection $\nabla_a-\mathcal Q_a\sharp$. Moreover, the right hand side of (\ref{DLexpandedtractor}) is projectively invariant.
\end{thm}
\begin{proof}The proof follows immediately from Theorem \ref{DLprop} and Lemma \ref{inv-remark} and from the computations above.
%$X^C\D_CL_{DEAB}=0$. 
%Since $L_{DEAB}$ is a parallel section for a connection it cannot have zeros unless it is zero. 
\end{proof}

% 
% To check that also the right-hand-side of equation (\ref{DLK}) vanishes we notice first that all expressions of the form $X^CW_{CA}\sharp$ vanish because of equation (\ref{XW0}). Next note that the expression of the form $X^C\D_C( \ldots )$ in the second line of (\ref{DLK})  vanishes because the argument has weight zero \tedz{correct?}.
%All other terms are of the form $X^C\D_AV_{CB \cdots}$ with $X^CV_{C\cdots}=0$. They are dealt with using relation (\ref{crucial}) above as follows
%\begin{eqnarray*}
%2 X^C\D_C(P\D^2K)_{DEAB}
%&=&
%\tfrac{1}{2}(
%W_{AD}\sharp K_{BE}+W_{AE}\sharp K_{BD})
%+\tfrac{3}{2} W_{B(A}\sharp K_{DE)}
%\\&&
%+\tfrac{1}{4}
%\left(
%W_{AB}\sharp K_{ED}
%+
%W_{ED}\sharp K_{AB}
%+ 
%2 W_{E(A}\sharp K_{B)D}+
%2
% W_{D(A}\sharp K_{B)E}\right)
%\\&&
%+\tfrac{1}{4}
%\left(
%W_{AB}\sharp K_{DE}
%+W_{DE}\sharp K_{AB}
%+
%2 W_{E(A}\sharp K_{B)D}+
%2 W_{D(A}\sharp K_{B)E}\right)
%\\
%&=&0.
%\end{eqnarray*}
%
%\bigskip
%

%
%\section{Applications}\label{app}
%
%Maybe we put something here
%
%
%\edz{RS: Add citations do with superintegrable systems }

%\bibliographystyle{abbrv}
%\bibliography{GEOBIB}

\providecommand{\MR}[1]{}\def\cprime{$'$} \def\cprime{$'$} \def\cprime{$'$}

\end{document}